\documentclass[11pt, oneside]{amsart}
\RequirePackage{amsmath}
\usepackage{amsthm}
\usepackage[utf8]{inputenc}
\usepackage{hyperref}
\usepackage{url}
\usepackage{nomencl}
\makenomenclature{}

\usepackage{mathtools}
\usepackage{enumerate}
\usepackage{cite}
\usepackage[shortlabels]{enumitem}
\usepackage{tikz-cd}
\usepackage{stmaryrd}
\usepackage{latexsym}
\usepackage{amsmath,amssymb,amsthm,amscd}
\usepackage{xcolor}
\usepackage[nolabel, final]{showlabels}
\title[$B$-pairs and Lubin--Tate Iwasawa theory]{Universal norms of partially de Rham $B$-pairs and perfectness in Lubin--Tate Iwasawa theory}

\author{Marvin Schneider}
\address{Ruprecht-Karls-Universität Heidelberg,
	Mathematisches Institut,Im Neuenheimer Feld 205, D-69120 Heidelberg}
\email{mschneider@mathi.uni-heidelberg.de}
\author{Rustam Steingart}
\address{Ruprecht-Karls-Universität Heidelberg,
	Mathematisches Institut,Im Neuenheimer Feld 205, D-69120 Heidelberg}
\email{rsteingart@mathi.uni-heidelberg.de}
\date{\today}

\theoremstyle{plain}
\newtheorem{thm}{Theorem}[section]
\newtheorem{lem}[thm]{Lemma}
\newtheorem{rem}[thm]{Remark}
\newtheorem{prop}[thm]{Proposition}
\newtheorem{cor}[thm]{Corollary}

\newtheorem{ex}[thm]{Example}
\newtheorem*{cor*}{Corollary}
\newtheorem{introtheorem}{Theorem}

\theoremstyle{definition}
\newtheorem{defn}[thm]{Definition}

\newcommand{\NN}{\mathbb{N}}

\newcommand{\Gal}{\operatorname{Gal}}
\newcommand{\Hom}{\operatorname{Hom}}

\newcommand{\ZZ}{\mathbb{Z}}

\newcommand{\QQ}{\mathbb{Q}}

\newcommand{\CC}{\mathbb{C}}

\newcommand{\fB}{\mathbf{B}}

\newcommand{\cR}{\mathcal{R}}

\newcommand{\dR}{\mathrm{dR}}

\newcommand{\id}{\operatorname{id}}

\newcommand{\Bdr}{\fB_{\mathrm{dR}}}

\DeclarePairedDelimiter\abs{\lvert}{\rvert}%
\DeclarePairedDelimiter\norm{\lVert}{\rVert}%
\begin{document}
\begin{abstract}
	We study the behaviour of partially de Rham Selmer groups of 
	$B$-pairs along abelian $p$-adic Lie extensions. In the cyclotomic case 
	we compare them with the universal norms studied by Berger. 
	For analytic Lubin--Tate $(\varphi,\Gamma)$-modules we prove that the $\Psi=1$ eigenspace 
	is coadmissible of the expected rank over the algebra of analytic distributions 
	and establish a comparison result between its Lie-algebra coinvariants and 
	the projective limit of the  partially de Rham Selmer groups of the attached $B$-pair. 
\end{abstract}
\maketitle

	\tableofcontents
	
	\section*{Introduction}

Let $E,K/\QQ_p$ be finite extensions and $K_\infty/K$ be a $p$-adic Lie extension with Galois group $\Gamma.$ Let $V$ be a finite dimensional $E$-vector space equipped with a continuous $E$-linear action of the absolute Galois group $G_K$ of $K$ and $T\subseteq V$ a $G_K$-stable lattice. 
An important invariant of $V$ is the Iwasawa cohomology
$$H^i_{\text{Iw}}(K_\infty/K,V):= (\varprojlim_{F} H^i(G_F,T) )[1/p],$$ where the projective limit runs over all finite intermediate extensions of $K_\infty/K$ and the transition maps are the corestriction. 
This defines a module over the completed group ring $\Lambda:= \mathcal{O}_E\llbracket \Gamma \rrbracket$ and plays an important role in the theory of $p$-adic $L$-functions. 
$E$-linear representations of $G_K$ can naturally be viewed as a subcategory of the category of $G_K$-equivariant vector bundles on the Fargues--Fontaine curve (cf.\cite{farguesfontainecourbe}). The datum of such a vector bundle is equivalent to giving a pair $(W_e,W_{\mathrm{dR}}^+)$ of free modules over the rings $E \otimes_{\QQ_p} \fB_e$ and $E \otimes_{\QQ_p}\fB_{\text{dR}}^+$ each equipped with continuous actions of $G_K,$ together with a comparison isomorphism $\Bdr \otimes_{\Bdr^+}W_{\text{dR}}^+ \cong \Bdr \otimes_{\fB_e} W_e.$ Which are precisely $B$-pairs previously introduced in \cite{Berger2008-vs} by Berger for $E=\QQ_p$ and studied extensively by Nakamura (cf. \cite{nakamura2009classification}) for general $E.$
Recall that representations ``coming from geometry'' are de Rham. The property of being de Rham is not stable under extensions. Instead an extension in $H^1_{cts}(K,V)$ is de Rham, if and only if it belongs to a certain Selmer group namely the kernel of the natural map $H^1_{cts}(G_K,V) \to H^1_{cts}(K,\Bdr \otimes_{\QQ_p}V),$ which is denoted by $H^1_g(K,V).$
In \cite{Berger2005} Berger determined (up to finite dimensional vector spaces) the subgroup $H^1_{\text{Iw},g}(V)$ consisting of classes which lie in $H^1_g(K_n,V)$ for every $n$ where $K_\infty = \bigcup K_n$ is the $p$-cyclotomic extension. 
Denoting by $\operatorname{Fil}^1V$ the maximal subrepresentation which has all Hodge--Tate weights $\geq1$ he shows that the natural (injective) map $H^1_{\text{Iw}}(\operatorname{Fil}^1V) \to H^1_{\text{Iw}}(V)$ has image contained in $H^1_{\text{Iw},g}(V)$ and the co-kernel of the induced map $$H^1_{\text{Iw}}(\operatorname{Fil}^1V) \to H^1_{\text{Iw},g}(V)$$ is finite dimensional. Recently these results have been generalised to twisted cyclotomic extensions (cf.\cite{Berger2026}). 
In the first part of the article we study the behaviour of certain \textit{partial} de Rham Selmer groups of $B$-pairs along $p$-adic Lie extensions.
The Galois cohomology of a $B$-pair can be defined as the total complex of
$$[C^\bullet_{cts}(G_K,W_e \oplus W_{\text{dR}}^+) \to C^\bullet_{cts}(G_K,W_{\text{dR}})]$$ using the usual continuous co-chain complexes. 
The partially de Rham Selmer groups can be defined if $E$ contains a normal closure of $K.$ In this case we have a natural decomposition $E \otimes_{\QQ_p}\Bdr \cong \prod_{\sigma} E \otimes_{\sigma,K}\Bdr$ indexed by the embeddings of $K$ into $E,$ which induces $W_{\text{dR}} \cong \prod_{\sigma} W_{\text{dR},\sigma}.$ We define $$H^1_{g,J}(K_n,W):= \operatorname{ker}\{H^1(K_n,W) \to \prod_{\sigma \in J} H^1(K_n,W_{\mathrm{dR},\sigma}) \}.$$ These were first introduced by Yiwen Ding in \cite{ding2017protect}. Recently in \cite{steingart2025partial} the second named author proved Duality and Euler--Characteristic results for certain cases of these groups (namely when the weights of $W$ at $J$ are non-positive).
When $W$ is not attached to a representation, we do not have an integral lattice, hence we do not have a $\Lambda$-module we can work with. Instead we will study the projective limit 
$\varprojlim_n H^1_{g,J}(K_n,W),$ which we view as a module over the algebra of smooth distributions $\mathcal{D}_\infty = \varprojlim_n E[\Gamma/\operatorname{Gal}(K_\infty/K_n)].$

Our first main result is the following 
\begin{introtheorem} (cf. Proposition \ref{prop:smoothrank})
	Let $W$ be an $E$-$B$-pair which is de Rham at $J \subset \Hom(K,E).$ Let
	$N = \sum_{\sigma \in J} \dim_E(H^0(G_K,W_{\mathrm{dR},\sigma}^+)) $. Then 
	\begin{enumerate}
		\item $H^1_{\text{Iw},g,J,\nabla}(W):= \varprojlim H^1_{g,J}(K_n,W)$ is a coadmissible $\mathcal{D}_\infty$-module of rank $[K:\QQ_p] \operatorname{rank}(W)-N.$
		\item The image of the natural map 
		$\varprojlim H^1(K_n,W(\delta)) \to \varprojlim H^1(K_n,W)$ is $H^1_{\text{Iw},g,J,\nabla}(W).$
	\end{enumerate}
\end{introtheorem} \footnote{The symbol $\nabla$ in the notation is suggestive. When $\Gamma$ is one-dimensional, and $\nabla$ is a basis of the Lie-algebra, one can show $\mathcal{D}_\infty \cong \mathcal{D}/\nabla.$}
Note that $\mathcal{D}_\infty$ is not a domain and ``rank'' of a module $X$ has to be understood as follows: We say $X$ is good, if the dimension of the fibre of $X$ at all but finitely many maximal ideals of $\mathcal{D}_\infty$ is constant. If $X$ is good, we denote this constant by $\operatorname{rank}_{\mathcal{D}_\infty}X.$ The $B$-pair $W(\delta)$ is a certain sub $B$-pair of $W$ which plays a similar  (but not the same) role as $\operatorname{Fil}^1V$ in Berger's result. 
When $J$ contains all embeddings and $K_\infty/K$ is the cyclotomic extension, we can compare with Bergers result directly to obtain the following result.
\begin{introtheorem} (cf. Proposition \ref{prop:comparisonberger})
	Let $K_\infty/K$ be the cyclotomic extension, let $V$ be de Rham and let 
	$N_0$ be the number of Hodge--Tate weights of $V$ which are $<1.$
	\begin{enumerate}[(i)]
		\item The natural map 
		 $\mathcal{D}_\infty \otimes_\Lambda H^1_{\text{Iw},g}(V) \to H^1_{\text{Iw},g,\nabla}(V)$ has finite-dimensional kernel. 
		 \item 
		 The cokernel is finite-dimensional if and only if $[K:\QQ_p]\dim V - N_0= [K:\QQ_p]\dim\operatorname{Fil}^1V,$
		 if and only if $[K:\QQ_p]\dim\operatorname{Fil}^1V$ is equal to the number of HT weights $\geq 1.$
	\end{enumerate}
\end{introtheorem}

A special case of partial de Rham Selmer groups is the $K$-analytic situation, where we fix an embedding $K \to E$ and declare $W$ to be $K$-analytic if it is de Rham of weight $0$ at all embeddings except maybe the fixed one. 
If $K_\infty/K$ is the Lubin--Tate extension and $K=E$, the results of Berger \cite{Berger2016} (in the étale case) and Poyéton \cite{Poyeton2023-sw} in general (together with the observation \cite[Proposition 2.3.5]{steingart2025partial} that our notion of $K$-analytic coincides with the à priori weaker notion of Berger, requiring only Hodge--Tate of weight $0$) allows to attach to $W$ a $(\varphi,\Gamma)$-module $M=M(W)$ over the Robba-ring $\cR_K,$ with coefficients in $K.$ (We refer the reader to Section \ref{section:analytic} for precise definitions).
In the cyclotomic case there is a comparison isomorphism 

$$\mathcal{D}\otimes_{\Lambda}H^1_{\text{Iw}}(V) \cong M(V)^{\Psi=1}$$ due to Colmez (cf. \cite{colmez1998theorie}), where $\Psi$ denotes the canonical left-inverse operator of $\varphi$ and $\mathcal{D}$ is the algebra of locally $\QQ_p$-analytic distributions. 
Naively, one would like to replace $\mathcal{D}$ by $\mathcal{D}^K$ the algebra of $K$-analytic distributions to obtain a similar statement for the Lubin--Tate case, but it turns out, that the naive version of this result can not be true and it has proven to be difficult to show that the right-hand side, i.e., $M^{\Psi=1}$ is even a coadmissible module in the sense of Schneider--Teitelbaum (cf. \cite{schneider2003algebras}). 

The main result of the second part of the article is the following 
\begin{introtheorem}
Let $M$ be a $K$-analytic $(\varphi_K,\Gamma_K)$-module over the Robba ring $\cR_K.$ Then 
$M^{\Psi=1}$ is a coadmissible, finitely presented $\mathcal{D}^K(\Gamma_n,K)$-module of rank 
$$[\Gamma_K\colon\Gamma_n]\operatorname{rank}(M).$$ In particular, the complex 
$$M \xrightarrow{\Psi-1} M $$ is a perfect complex of $\mathcal{D}^{K}(\Gamma_K,K)$-modules.
\end{introtheorem}

The proof of this theorem is quite technical and we give an outline of the main steps.
First we show that the statement is true for $(\varphi_K,\Gamma_K)$-modules over the Robba ring $\cR_{\CC_p}.$ The technical foundations for this were laid out in \cite{steingart2024iwasawa}  and make use of Fourier Theory. Over $\CC_p$ the distribution algebra is isomorphic to the functions on an open unit disc and we can, by writing the unit disc as a union of closed discs, express it as a projective limit $\mathcal{D}^K = \varprojlim D_n$ of PIDs and one can construct a certain family $\mathbf{Dfm}_n = M\hat{\otimes}_K D_n$ of  $(\varphi_K,\Gamma_K)$-modules over $D_n,$ for which the analytic cohomology $C_{an}(\Psi \times \Gamma_n,\mathbf{Dfm}_n)$ forms a projective system of perfect complexes. It was shown in \cite{steingart2024iwasawa} that the latter has coadmissible cohomology and that for $M_{\CC_p} = \CC_p \hat{\otimes}_KM$, if $(M_{\CC_p})^{\Psi=1}$ is coadmissible, then the natural map $M_{\CC_p}^{\Psi=1} \to H^1(\mathbf{R}\lim C_{an}(\Psi \times \Gamma_n,\mathbf{Dfm}_n))$ is an isomorphism. With some more care it is in fact possible to show directly, that this map is an isomorphism which in turn forces $M_{\CC_p}^{\Psi=1}$ to be coadmissible (cf. Theorem \ref{thm:psicoadmissible}). It is at this stage also possible to see, that the rank of $M_{\CC_p}^{\Psi=1}$ is equal to the additive inverse of the Euler--Poincaré characteristic of the analytic cohomology (over $\CC_p$), but it is not clear that  the value has to be $-\operatorname{rank}(M)!$ This is where Partial Selmer groups come in as a key input. We would like to remark that a similar thing happens already in the cyclotomic case. For example, \cite{KPX} rely on reducing to the étale case to determine the rank of $M^{\Psi=1}$ from Tate's formula for the Euler--Poincaré characteristic, while e.g. Herr in \cite{herr1998cohomologie} can easily show that the Euler--Poincaré characteristic of the Herr complex is equal to the rank of $M^{\Psi=1}$ but then requires substantially more work to compute the latter rank (independently of the results of Tate).
Using partial Selmer groups it was shown in \cite{steingart2025partial} that for a $(\varphi_K,\Gamma_K)$-module $M_K$ over $\cR_K$ the Euler--Characteristic of its analytic cohomology is $-\operatorname{rank}(M_K).$
Using a ``faithfully flat descent'' argument we show that $M_K^{\Psi=1}$ is coadmissible (and finitely presented) as a consequence of the finite presentability of $\CC_p \hat{\otimes}_KM^{\Psi=1}.$
A naive Galois-descent argument runs into problems because already $H^1(G_K,\CC_p)\neq0.$ Instead our descent argument uses only functional analysis. 
The coadmissibility of $M_K^{\Psi=1}$ allows us to show a base change result which implies that the Euler--Poincaré characteristics of the analytic cohomology of $M_K$ and $M_{\CC_p}$ coincide. 
We go on to show, that if $M = M(W)$ and $J$ contains all embeddings except the identity, then the natural map 
$$M^{{\Psi=1}}/\nabla \to \varprojlim_{n}H^1_{g,J}(K_n,W) = H^1_{\text{Iw},g,J,\nabla}(W)$$ is injective and has finite-dimensional cokernel, which gives another description of the smooth partial Selmer groups considered in the first part. 
Finally, we revisit a comparison map, which is essential in applications in the work of  Manji (cf. \cite{manji2024iwasawa}) on an analytic Main conjecture, which in turn were building on some conjectural results of \cite{steingart2024iwasawa} , which can now be proven unconditionally.  
If $V^*(1)$ is a $K$-analytic Galois representation, then one can construct a natural map 
$$\text{comp} \colon H^1_{\text{Iw}}(V) \to M^{\Psi=1},$$ where $M = \operatorname{Hom}_{\cR_K}(M(V^*(1)),\cR_K(\chi_C))$ is the analytic $(\varphi,\Gamma)$-module dual (with respect to the residue pairing) to $M(V^*(1)),$ where $\cR_K(\chi_C)$ denotes the (non-étale if $K \neq \QQ_p$) rank one module attached to the character $x\abs{x}.$  We show that the image of this map generates $M^{\Psi=1}$ as a module over the distribution algebra. 
\section*{Acknowledgements}

	 We are greatly indebted to Muhammad Manji for countless discussions over the years. His applications of conjectural results from the second named author's thesis have been a constant motivation to revisit the topic. 
	  The research was supported by the Deutsche Forschungsgemeinschaft (DFG) through the Collaborative
Research Centre TRR 326 \textit{Geometry and Arithmetic of Uniformized Structures} project number
444845124. 
\section*{AI Disclosure}
	 The article is written by the authors without the use of AI tools for its mathematical content. The authors have used AI tools for spellchecking and copy-editing.
\section{General Notation and preliminaries}
We fix $K /\QQ_p$ finite and $E$ a finite extension of $\QQ_p$ containing a normal closure of $K.$ Let $G_K := \Gal(\bar{K}/K)$ be the absolute Galois group of $K.$ We fix $K_\infty = \bigcup K_n$ a $p$-adic Lie extension of $K$ with Galois group $\Gamma_K := \Gal(K_\infty/K)$ and where $K_n/K$ is Galois with finite Galois group $Q_n.$ We write $\Gamma_n:= \Gal(K_\infty/K_n).$ Implicitly we view all fields as subfields of the completion $\CC_p$ of an algebraic closure of $\QQ_p.$
\[\begin{tikzcd}
	& {\mathbb{C}_p} \\
	{K_\infty} \\
	{K_n} & E \\
	K \\
	{\mathbb{Q}_p}
	\arrow[no head, from=2-1, to=1-2]
	\arrow["{\Gamma_n}", no head, from=2-1, to=3-1]
	\arrow["{\Gamma_K}"', bend right=40, no head, from=2-1, to=4-1]
	\arrow["{Q_n}", no head, from=3-1, to=4-1]
	\arrow[no head, from=3-2, to=1-2]
	\arrow["{\text{finite}}", no head, from=3-2, to=4-1]
	\arrow["{\text{finite}}", no head, from=5-1, to=4-1]
\end{tikzcd}\]
For the moment, we will not assume that $\Gamma_K$ carries the structure of $K$-analytic Lie group. 
\subsection{Iwasawa and distribution algebras}
For any compact $p$-adic Lie group $G$ we denote by $\mathcal{D}(G,E)$ the algebra of $E$-valued locally analytic distributions on $G$ and, if $G$ is a $K$-analytic group we denote by $\mathcal{D}^K(G,E)$ the algebra of $E$-valued $K$-analytic distributions on $G.$ Recall, that $\mathcal{D}(G,E)$ is a Fréchet--Stein algebra in the sense of \cite{schneiderteitlbaumlocallyanalytic} meaning that it can be written as a projective limit of Banach algebras $\mathcal{D}(G,E) = \varprojlim_n \mathcal{D}_n(G,E)$ such that the transition maps $\mathcal{D}_{n+1}(G,E) \to \mathcal{D}_n(G,E)$ are flat. A $\mathcal{D}(G,E)$-module $M$ is called coadmissible if it is of the form $M = \varprojlim_n M_n$ where $M_n$ is a finitely generated $\mathcal{D}_n(G,E)$-module and the natural map $\mathcal{D}_n(G,E) \otimes_{\mathcal{D}_{n+1}(G,E)} M_{n+1} \to M_n$ is an isomorphism.
The algebra $\mathcal{D}(G,E)$ contains the Iwasawa algebra $\Lambda(G) := o_E\llbracket G \rrbracket$ and it turns out that $E \otimes_{o_E}\Lambda(G)$ is dense in $\mathcal{D}(G,E)$ and the ring extension is faithfully flat (cf. \cite[Lemma 3.1]{schneiderteitlbaumlocallyanalytic} and \cite[Theorem 5.2]{schneider2003algebras}).
\begin{lem} \label{lem:smoothdistr} Let $G$ be a compact $p$-adic Lie group 
	Let $J \subset \mathcal{D}=\mathcal{D}^{\QQ_p}(G,E)$ be the ideal generated by the Lie Algebra of $G,$ then the natural map 
	$$ \mathcal{D}_\infty:=\mathcal{D}/J \to \varprojlim_{H \triangleleft_o G} E[G/H],$$ where the limit is taken over all open normal subgroups $H$ of $G$ is an isomorphism.
\end{lem}
\begin{lem} \label{lem:projectivelimitofgrouprings} Suppose $\Gamma$ is abelian. Let $\mathcal{D}_\infty= \varprojlim E[\Gamma/\Gamma_n]$ and for an irreducible representation $\rho \colon \Gamma \to \CC_p^\times$ with open kernel we denote by $E_\rho$ the residue field of $\mathcal{D}_\infty$ at $\rho$ (viewed as a maximal ideal). The natural map 
	$$\mathcal{D}_\infty \cong \prod_{\rho} E_\rho$$ is an isomorphism, where $\rho$ runs through all conjugacy classes of irreducible representations ( i.e. characters \footnote{Here we allow characters valued in extensions of $E.$}) with open kernel. For any (finite dimensional) $E$-vector space $V$ with continuous and $E$-linear $G_K$-action
	this induces an isomorphism 
	$$\varprojlim H^1(K,E[\Gamma/\Gamma_n] \otimes V) \cong H^1(K,\mathcal{D}_\infty \otimes V) \cong \prod_{\rho}H^1(K,\rho \otimes V) ,$$ where $\rho$ is the one dimensional $E_\rho$-vector space with $E_\rho$-linear action of $G_K$ via $\rho$ pre-composed with the projection onto $\Gamma$.
\end{lem}
\begin{proof}
	Since $\operatorname{char}(E)=0$ we have that $E[\Gamma/\Gamma_n]$ is semisimple by Maschke's Theorem. Because $\Gamma$ is abelian, it follows that every irreducible representation is a character and appears with multiplicity one in the Artin--Wedderburn decomposition $E[\Gamma/\Gamma_n] = \prod_{\rho} M_{d_\rho}(\operatorname{End}(\rho)^{op}).$ In particular $E_\rho=\operatorname{End}(\rho)$ is indeed a field. For the second part use that $V$ is finite dimensional to commute the limit with the tensor product. 
\end{proof}
\begin{prop}\label{prop:coad1} Suppose $\Gamma$ is abelian.
	Let $e_\rho$ be the idempotent corresponding to $\rho$ in the decomposition of $\mathcal{D}_\infty.$ 
	A $\mathcal{D}_\infty$-module $X$ is coadmissible if and only if 
	\begin{enumerate}
		\item $e_\rho X$ is finite dimensional for all $\rho$
		\item The natural map $X \to \prod_{\rho} e_\rho X$ is an isomorphism
	\end{enumerate}
\end{prop}
\begin{proof}
Let $\mathcal{D}_n := E[\Gamma/\Gamma_n]$. By Lemma \ref{lem:projectivelimitofgrouprings}, $\mathcal{D}_n \cong \prod_{\rho} E_\rho$ indexed over characters $\rho$ factoring through $\Gamma/\Gamma_n$. 
Any finitely generated $\mathcal{D}_n$-module $X_n$ thus decomposes into finite-dimensional $E_\rho$-vector spaces as $X_n \cong \prod_{\rho} e_\rho X_n$. 
If $X \cong \varprojlim_n X_n$ is a coadmissible $\mathcal{D}_\infty$-module, the structural isomorphisms $\mathcal{D}_n \otimes_{\mathcal{D}_{n+1}} X_{n+1} \cong X_n$ 
force $e_\rho X_{n+1} \cong e_\rho X_n$ for all $\rho$ factoring through $\Gamma/\Gamma_n$. 
The inverse limit therefore yields $X \cong \prod_\rho e_\rho X$ with each component finite-dimensional, satisfying (1) and (2). 
Conversely, assuming (1) and (2), define $X_n := \prod_{\rho} e_\rho X$ restricted to characters factoring through $\Gamma/\Gamma_n$. 
Because this product is finite and each $e_\rho X$ is finite-dimensional, $X_n$ is finitely generated over $\mathcal{D}_n$. The natural projections $X_{n+1} \to X_n$ 
induce $\mathcal{D}_n \otimes_{\mathcal{D}_{n+1}} X_{n+1} \cong X_n$, exhibiting $X \cong \varprojlim_n X_n$ as a coadmissible $\mathcal{D}_\infty$-module.
\end{proof}
\begin{defn}
A coadmissible $\mathcal{D}_\infty$-module $X$ is called \textbf{good} if the dimension $\dim_{E_\rho} e_\rho X$ is constant for all but finitely many $\rho.$
\end{defn}
\begin{lem} \label{lem:goodexactsequence} The following hold. 
	\begin{enumerate}
	\item Finite $E$-dimensional $\mathcal{D}_\infty$-modules are good.
	\item Let $0 \to X' \to X \to X'' \to 0$ be an exact sequence of coadmissible $\mathcal{D}_\infty$-modules. Then if two out of the three modules are good, so is the third.
	\end{enumerate}
\end{lem}
\begin{proof}
	The first point is clear. 
For the second point use that $E_\rho$ is flat as a $\mathcal{D}_\infty$-module, hence, the induced sequences 
$0 \to e_\rho X' \to e_\rho X \to e_\rho X'' \to 0$ remain exact. 
\end{proof}
Henceforth and for the remainder of the article we will assume that $\Gamma$ is abelian.
\begin{prop} \label{prop:goodcoad}
Let $X = \prod_{\rho}X_\rho$ be a coadmissible $\mathcal{D}_\infty$-module. The following are equivalent. 

\begin{enumerate}[(i)]
\item $X$ is good.
\item $ \liminf \dim_{E_\rho} X_\rho$ and $\limsup \dim_{E_\rho} X_\rho$ are finite and coincide.
\item There exists a map $\mathcal{D}_\infty^s \to X$ with finite-dimensional kernel and co-kernel.
\end{enumerate}
\end{prop}
\begin{proof} 
	The equivalence between \textit{(i)} and \textit{(ii)} is clear. The implication \textit{(iii)} implies \textit{(i)} is a direct consequence of \ref{lem:goodexactsequence}. We show that \textit{(i)} implies \textit{(iii)}.
Write $X = \prod_{\rho}X_\rho,$ by assumption there exist in each component $(v^{\rho}_1,\dots,v^{\rho}_s),$ which are an $E_\rho$-basis at all but finitely many $\rho.$ Let $Y \subseteq X$ be the $D_\infty$-span of the diagonal elements $v_i:=(v_i^\rho)_\rho.$ Since both $\mathcal{D}^s$ and $X$ are coadmissible, we can apply the characterisation of coadmissible $\mathcal{D}_\infty$-modules from \ref{prop:coad1} to see that the kernel and co-kernel of the induced map $\mathcal{D}^s \to Y \to X$ are finite dimensional. 
\end{proof}
\begin{defn}
Let $X$ be a good coadmissible $\mathcal{D}_\infty$-module. We define the $\operatorname{rank}_{\mathcal{D}_\infty}X$ to be the common value of $\dim_{E_\rho} e_\rho X$ for all but finitely many $\rho.$
\end{defn}
\begin{rem}
The proof of Proposition \ref{prop:goodcoad} shows that there exists a map $\mathcal{D}_\infty^s \to X$ with finite-dimensional kernel and co-kernel, where $s=\operatorname{rank}_{\mathcal{D}_\infty}X.$
\end{rem}

\begin{prop} \label{prop:DmodNablaLimit}
	Let $M$ be a finitely generated $\Lambda_E = o_E\llbracket \Gamma \rrbracket[1/p]$-module let $I_n$ be the kernel of the natural map $\Lambda \to E [\Gamma/\Gamma_n]$, then the natural map $$\mathcal{D}/\operatorname{Lie}(\Gamma) \otimes_\Lambda M \cong \varprojlim M/I_nM,$$ is an isomorphism.
\end{prop}
\begin{proof}
	When $M$ is finitely generated free the statement is clear by Lemma \ref{lem:smoothdistr}.
	We denote by $\mathcal{I}_n$ the kernel of the natural map $\mathcal{D} \to E[\Gamma/\Gamma_n].$ Because $\Lambda$ is noetherian, we can find a finite presentation of $M$ of the form 
	$$\Lambda_E^d \to \Lambda_E^{d'} \to M \to 0.$$
	Tensoring with $\mathcal{D}/\mathcal{I}_n (= E[\Gamma/\Gamma_n])$ we obtain an exact sequence 
	$$ (\mathcal{D}/\mathcal{I}_n)^d\to (\mathcal{D}/\mathcal{I}_n)^{d'} \to M/I_nM \to 0$$
	using 
	$$\mathcal{D}/\mathcal{I}_n \otimes_{\Lambda_E} M = E[\Gamma/\Gamma_n] \otimes_{\Lambda_E} M = \Lambda_E/I_n \otimes_{\Lambda_E}M = M/I_nM.$$
	Let $$X_n := \operatorname{ker}((\mathcal{D}/\mathcal{I}_n)^{d'}\to M/I_nM)$$ such that we have a short exact sequence
	\begin{equation} \label{eq:seqXM}0 \to X_n \to (\mathcal{D}/\mathcal{I}_n)^{d'} \to M/I_nM \to0 \end{equation}
	and let 
$$ Y_n = \operatorname{ker}((\mathcal{D}/\mathcal{I}_n)^d\to (\mathcal{D}/\mathcal{I}_n)^{d'})
$$
	such that we have a short exact sequence 
\begin{equation}\label{eq:seqYX}0 \to Y_n \to (\mathcal{D}/\mathcal{I}_n)^d \to X_n \to0.
\end{equation}
The $X_n$ and $Y_n$ are finite $E$-dimensional vector spaces and hence form Mittag-Leffler systems. 
Applying $\varprojlim_n$ to the sequence \eqref{eq:seqYX} shows
$$\varprojlim_n (D/\mathcal{I}_n)^d \to \varprojlim_n X_n \to 0$$ is exact and 
applying $\varprojlim_n$ to the sequence \eqref{eq:seqXM} now yields an exact sequence
$$\varprojlim_n (D/\mathcal{I}_n)^d \to \varprojlim_n(\mathcal{D}/\mathcal{I}_n)^{d'} \to \varprojlim_n M/I_nM \to0$$
Tensoring the presentation of $M$ with $\mathcal{D}_\infty$ and comparing with the above sequence yields a commutative diagram with exact rows
\[\begin{tikzcd}
	{\mathcal{D}_\infty^d} & {\mathcal{D}_\infty^{d'}} & {\mathcal{D}_\infty \otimes_{\Lambda_E}M} & 0 \\
	{\varprojlim_n(\mathcal{D}/\mathcal{I}_n)^d} & {\varprojlim_n(\mathcal{D}/\mathcal{I}_n)^{d'}} & {\varprojlim_n M/I_nM} & 0
	\arrow[from=1-1, to=1-2]
	\arrow["\cong", from=1-1, to=2-1]
	\arrow[from=1-2, to=1-3]
	\arrow["\cong", from=1-2, to=2-2]
	\arrow[from=1-3, to=1-4]
	\arrow[from=1-3, to=2-3]
	\arrow[from=2-1, to=2-2]
	\arrow[from=2-2, to=2-3]
	\arrow[from=2-3, to=2-4]
\end{tikzcd}\]
From the free case we see that the first and the second vertical maps are isomorphisms, the claim now follows from the Five Lemma.
\end{proof}

This result can be generalised to higher Tor groups in the case $\Gamma \cong \ZZ_p^d.$ 
\begin{lem} \label{lem:torRlim} Let $\Gamma \cong \ZZ_p^d.$
Let $M$ be a finitely generated $\Lambda[1/p]$-module, let $\mathcal{M}:= \mathcal{D} \otimes_\Lambda M$ and let $\mathcal{M}_\infty:= \mathcal{D}_\infty \otimes_{\Lambda} M$, then
$$\operatorname{Tor}_i^\mathcal{D}(\mathcal{D}_\infty,\mathcal{M}) \cong \varprojlim_n \operatorname{Tor}_i^{\Lambda[1/p]}(\Lambda[1/p]/I_n,M)$$ for every $i.$
\end{lem}
\begin{proof}
Since $\mathcal{D}$ is flat over $\Lambda$ we have 
$$\operatorname{Tor}_i^{\Lambda[1/p]}(\Lambda[1/p]/I_n,M) \cong \operatorname{Tor}_i^\mathcal{D}(\mathcal{D}/\mathcal{I}_n,\mathcal{M}).$$
By resolving $\mathcal{D}/\mathcal{I}_n$ using the Koszul complex we can see that these Tor groups are finitely generated $\mathcal{D}/I_n$-modules. In particular, they are finite dimensional as $E$-vector spaces. By Mittag--Leffler $\varprojlim_n^1\operatorname{Tor}_{i+1}^\mathcal{D}(\mathcal{D}/\mathcal{I}_n,\mathcal{M})=0,$ which implies that 
$$ H^{-i}(\mathbf{R}\lim (\mathcal{D}/\mathcal{I}_n \otimes_{\mathcal{D}}^{\mathbb{L}}\mathcal{M}) \cong \varprojlim_n\operatorname{Tor}_i^{\Lambda[1/p]}(\Lambda[1/p]/I_n,M)$$
By similar Koszul complex arguments the complex $\mathcal{D}_\infty \otimes_{\mathcal{D}}^{\mathbb{L}}\mathcal{M}$ has coadmissible cohomology groups. Furthermore, by representing the complex using a flat resolution of $M$ instead of a resolution of $\mathcal{D}_\infty$ we get a complex of $\mathcal{D}_\infty$-modules with coadmissible cohomology. 
By \cite[Proposition  3.15 ]{steingart2024iwasawa}  (see also Lemma \ref{lem:consistent})
we have 

$$\mathcal{D}_\infty \otimes_{\mathcal{D}}^{\mathbb{L}}\mathcal{M} \cong \mathbf{R}\lim_n[\mathcal{D}/\mathcal{I}_n \otimes_{\mathcal{D_\infty}}^{\mathbb{L}} \mathcal{D}_\infty \otimes_{\mathcal{D}}^{\mathbb{L}}\mathcal{M} ].$$ This implies the claim since $\mathcal{D}/\mathcal{I}_n$ is flat over $\mathcal{D}_\infty$ and hence $$\mathcal{D}/\mathcal{I}_n \otimes_{\mathcal{D_\infty}}^{\mathbb{L}} \mathcal{D}_\infty \cong \mathcal{D}/\mathcal{I}_n.$$

\end{proof}

\section{Universal norms in Partial de Rham Selmer groups of $B$-pairs}
\subsection{Preliminaries on $B$-pairs}

We recall from \cite{nakamura2009classification} the notion of an $E$-$B$-pair. An $E$-$B$-pair is a pair $W = (W_e,W_{\mathrm{dR}}^+)$ where $W_e$ is a finite free $\fB_e \otimes_{\QQ_p} E$-module and $W_{\mathrm{dR}}^+$ is a finite free $\fB_{\mathrm{dR}}^+ \otimes_{\QQ_p} E$-module each equipped with a continuous action of $G_K,$ trivial on $E$ and semi-linear with respect to the natural action on $\fB_e$ and $\fB_{\mathrm{dR}}^+$ respectively, together with a $G_K$-equivariant isomorphism
$$\fB_{\mathrm{dR}} \otimes_{\Bdr^+} W_{\mathrm{dR}}^+ \cong \fB_{\mathrm{dR}} \otimes_{\fB_e}W_e.$$ The rank of $W$ is defined to be the rank of $W_e$ as an $\fB_e \otimes_{\QQ_p} E$-module.
We denote by $H^i(G_K,W)$ the cohomology of the total complex of continuous cochains $C^\bullet_{cts}(G_K,[W_e \oplus W_{\mathrm{dR}}^+ \to W_{\mathrm{dR}}])$ (shifted appropriately such that $H^0(G_K,W) = H^0(G_K,W_e)\cap H^0(G_K,W_{\mathrm{dR}}^+)$). The category of $E$-$B$-pairs is equivalent to the category of $(\varphi,\Gamma)$-modules over the ring $E \otimes_{\QQ_p}\fB^{\dagger}_{\text{rig},K}$ (cf. \cite[Theorem 1.36]{nakamura2009classification}), whose definition we do not recall, and contains the category of $E$-linear representations of $G_K$ as a full subcategory, where the $E$-$B$-pair attached to a representation $V$ is given by $(\fB_e \otimes_{\QQ_p} V, \fB_{\mathrm{dR}}^+ \otimes_{\QQ_p} V).$ Another equivalent way of thinking of $E$-$B$-pairs is as equivariant vector bundles on the Fargues--Fontaine curve (using the notation of \cite{farguesfontainecourbe}) $X_{\CC_p^\flat,E} \cong E \times_{\QQ_p}X_{\CC_p^\flat,\QQ_p}$ (with $G_K$ acting trivially on $E$). Since $E$ contains a normal closure of $K$ and $\Bdr^+$ contains an algebraic closure of $K,$ we have a decomposition 
$$W_{\dR}^+ = \prod_{\sigma \in \Hom(K,E)} W_{\mathrm{dR},\sigma}^+$$ and similarly for $W_{\dR} = \Bdr \otimes_{\Bdr^+} W_{\dR}^+.$
We say that $W$ is de Rham at $\sigma$ if $H^0(G_K,W_{dR,\sigma})$ has $E$-dimension equal to the rank of $W.$ As usual, we say that $j$ is a Hodge--Tate weight of $W$ at $\sigma$ if $H^0(G_K,t^{-j}W_{\mathrm{dR},\sigma}^+/t^{-j+1}W_{\mathrm{dR},\sigma}^+) \neq 0.$ With this convention the cyclotomic character has Hodge--Tate weight $1.$
\subsection{Partial Selmer groups}
Let $W = (W_e,W_{\mathrm{dR}}^+)$ be an $E$-$B$-pair. Assume that $W$ is de Rham at a set $J \subset \Hom(K,E).$
Let $\delta = \prod_{\sigma \in J} \sigma^{\kappa_{\sigma}}$ where $\kappa_{\sigma}$ is is chosen large enough to ensure $H^0(G_K,t^{\kappa_{\sigma}}W_{\mathrm{dR},\sigma}^+)=0.$
\begin{lem} \label{lem:basicsequence} Suppose $J \neq \emptyset.$
	There is a long exact sequence 
	\begin{align*}
	0 &\to H^0(G_K,W) \to \bigoplus_{\sigma\in J} H^0(G_K,W_{\mathrm{dR},\sigma}^+) \to H^1(G_K,W(\delta)) \\
	 &\to H^1(G_K,W) \to \bigoplus H^1(G_K,W_{\mathrm{dR},\sigma}^+) \to \dots.
	\end{align*}
and the image of the natural map 
	$$H^1(G_K,W(\delta)) \xrightarrow{\operatorname{c_\delta}} H^1(G_K,W)$$
	is 
	$$H^1_{g,J}(W) = \operatorname{ker}[H^1(G_K,W) \to \prod_{\sigma \in J} H^1(G_K,W_{\mathrm{dR},\sigma})].$$
	In particular, we have an exact sequence
	$$0 \to H^0(G_K,W) \to \bigoplus_{\sigma\in J} H^0(G_K,W_{\mathrm{dR},\sigma}^+) \to H^1(G_K,W(\delta)) \to H^1_{g,J}(G_K,W) \to 0.$$
\end{lem}
\begin{proof}
See \cite[Equation (7) and the preceding discussions]{ding2017protect}.
\end{proof}
\begin{lem} \label{lem:dimensionformula} Let $N:= \sum_{\sigma \in J} \dim H^0(G_K,W_{\mathrm{dR},\sigma}^+).$
We have $$\dim H^1_{g,J}(G_K,W) = [K:\QQ_p] \operatorname{rank}(W) - N + \dim H^2(G_K,W(\delta)) + \dim H^0(G_K,W). $$
If the weights of $W$ at $\sigma$ are non-positive, then $N = \operatorname{rank}(W) \cdot |J|$ and the above formula simplifies to
$$\dim H^1_{g,J}(G_K,W) = ([K:\QQ_p]-|J|) \operatorname{rank}(W) + \dim H^2(G_K,W(\delta)) + \dim H^0(G_K,W).$$
\end{lem}

\subsection{Asymptotic behaviour of $H^1_{g,J}.$}
We now turn to studying the behaviour of $H^1_{g,J}(G_{K_n},W)$ for $n \in \NN.$ Let $Q_n:= G_K/G_{K_n}\cong \Gamma/\Gamma_n.$ The following remark is a direct consequence of Shapiro's Lemma 
\begin{rem} 
We have $H^i(K,E[Q_n] \otimes_E W) \cong H^i(K_n,W).$
\end{rem}
We can naturally think of $E [Q_n] \otimes W$ as an $E$-$B$-pair. The dimension formula of Lemma \ref{lem:dimensionformula} then implies that
\begin{align*}\dim H^1_{g,J}(K_n,W) &= ([K:\QQ_p]) \operatorname{rank}(W) \cdot |Q_n| - N^{(n)}\\
	&+ \dim H^2(G_K,E[Q_n] \otimes W(\delta)) + \dim H^0(G_K,E[Q_n] \otimes W).\end{align*}
where $N^{(n)} = \sum_{\sigma \in J} \dim H^0(G_K,E[Q_n] \otimes W_{\mathrm{dR},\sigma}^+).$
We will show that the dimensions of the $H^0$ and $H^2$ terms are bounded independently of $n$ while the dimensions of the first two summands grow by a factor of $|Q_n| = [K_n:K].$
This will allow us to determine the rank of 
$$\varprojlim_n H^1_{g,J}(K_n,W)$$ as a module over the smooth distribution algebra. 
\begin{lem} We have 
\begin{enumerate}[(i)]
	\item $$\dim_E H^0(G_K,E[Q_n] \otimes_E W_{\mathrm{dR},\sigma}^+) = [K_n:K] \dim_E H^0(G_{K}, W_{\mathrm{dR},\sigma}^+)$$
	\item We have $N^{(n)} = [K_n:K] \cdot N.$
\end{enumerate}
\end{lem}

\begin{proof}
	Let $X:= W_{\mathrm{dR},\sigma}^+$
	We rewrite 
	$E[Q_n] = E \otimes_{\QQ_p}\QQ_p[Q_n]$ and since
	$$E[Q_n] \otimes_{E}X = E \otimes_{\QQ_p} \QQ_p[Q_n] \otimes_{E} X = 
	\QQ_p[Q_n] \otimes_{\QQ_p} E \otimes_E X = \QQ_p[Q_n] \otimes_{\QQ_p}X$$ it suffices to show 
	 $$\dim_E H^0(G_K,\QQ_p[Q_n] \otimes_{\QQ_p} X) = [K_n:K] \dim_E H^0(G_{K}, X).$$
To this end we can rewrite 
$$ \QQ_p[Q_n] \otimes_{\QQ_p} X = \QQ_p[Q_n] \otimes_{\QQ_p} \Bdr^+ \otimes_{\Bdr^+}X.$$
Now observe, that $\QQ_p[Q_n]$ is a smooth representation, in particular it is $\overline{K}$-admissible. Since $\Bdr^+$ contains an algebraic closure of $K,$ it follows that 
$\QQ_p[Q_n] \otimes_{\QQ_p} \Bdr^+ \cong (\Bdr^+)^{[K_n:K]}$ as representations. 
Hence $$ \QQ_p[Q_n] \otimes_{\QQ_p} X \cong X^{[K_n:K]}$$ as representations which implies the claim. 
\end{proof}
\begin{rem} \label{rem:HTweightsless1}
If $V$ is a de Rham $E$ representation and $$W_{\mathrm{dR}}^+:= \Bdr^+ \otimes_{\QQ_p}V$$, then $$N = \sum_{\sigma \in J} \dim_E(H^0(G_K,W_{\mathrm{dR},\sigma}^+)) $$ is the number of non-positive HT-weights of $V$ at the embeddings in $J$ (counted with multiplicities).
\end{rem}
Let
$$c_{\delta}^{(n)} \colon H^1(G_{K_n},W(\delta)) \xrightarrow{\operatorname{c_\delta}} H^1(G_{K_n},W).$$

\begin{lem} \label{lem:trivialH0}
	Let $W = (W_e,W_{\mathrm{dR}}^+)$ be an $E$-$B$-pair. Then for all but finitely many isomorphism classes $\rho$ of finite order characters $\rho \colon G_K \to \Gamma_K \to \overline{\QQ_p}^\times,$ the group $H^0(G_K,W\otimes_E E_\rho(\rho))$ is trivial.
\end{lem}
\begin{proof} Suppose that $\rho$ is trivial on $\Gamma_n.$
	We denote by $E_\rho$ the residue field of $E[\Gamma/\Gamma_n]$ at the maximal ideal corresponding to $\rho.$ Observe that $E_\rho(\rho)$ is irreducible as an $E$-representation. 
Note that $H^0(G_K,W\otimes \rho)$ is a subspace of $H^0(G_K,W_e\otimes \rho) = H^0(G_K,\Hom(\rho^*, W_e)).$
	By irreducibility any such homomorphism is either $0$ or injective. And since the kernel of $\rho$ is open, this means that there exists an open subgroup $G_{K'}$ such that $W_e$ contains a $G_{K'}$-invariant subspace of dimension $\geq \dim_E{E_\rho}.$
	But $\dim_E H^0(K',W_e) \leq \operatorname{rank}(W_e)$ is bounded independently of $K'$ (cf. \cite[Corollary 2.14]{berger2021triangulable}). 
	Hence, as soon as the degree of $E_\rho$ exceeds this dimension, no such embedding can exist, but the field $E_\rho$ contains the image of $\rho.$ We conclude that there exists $n \gg 0 $ such that $H^0(K,\rho \otimes W) \neq 0 $ implies that $\rho$ factors over $\Gamma/\Gamma_n,$ but there are only finitely many $\rho$ of this form. 
\end{proof}
\begin{prop} \label{prop:smoothrank}
	Let $W$ be an $E$-$B$-pair which is de Rham at $J \subset \Hom(K,E).$ Let
	$N = \sum_{\sigma \in J} \dim_E(H^0(G_K,W_{\mathrm{dR},\sigma}^+)) $. Then 
	\begin{enumerate}
		\item $H^1_{\text{Iw},g,J,\nabla}(W):= \varprojlim H^1_{g,J}(K_n,W)$ is a coadmissible $\mathcal{D}_\infty$-module of rank $[K:\QQ_p] \operatorname{rank}(W)-N.$
		\item The image of the natural map 
		$\varprojlim H^1(K_n,W(\delta)) \to \varprojlim H^1(K_n,W)$ is $H^1_{\text{Iw},g,J,\nabla}(W).$
	\end{enumerate}
\end{prop}
\begin{proof} 
	
	By \ref{lem:trivialH0}, we have that $H^0(G_K,E[Q_n] \otimes W)$ is bounded independently of $n,$ since all but finitely many components are trivial. Dually the same holds for the $H^2$ term. Hence the only contribution comes from the first two terms in the dimension formula of Lemma \ref{lem:dimensionformula}. Decomposing into character components and applying the formula componentwise shows that $H^1_{\text{Iw},g,J,\nabla}(W)$ is good and the rank of $H^1_{\text{Iw},g,J,\nabla}(W)$ is $[K:\QQ_p] \operatorname{rank}(W)-N.$ 
 For the second part we have on every level a surjection
$$H^1(G_K,E[\Gamma/\Gamma_n] \otimes W(\delta)) \to H^1_{g,J}(G_K,E[\Gamma/\Gamma_n] \otimes W)$$ by Lemma \ref{lem:basicsequence}.
The kernels at each level are finite dimensional and hence form a Mittag-Leffler system, which implies the claim.
\end{proof}
\begin{rem} The formula from Proposition \ref{prop:smoothrank} takes a simpler shape, if all weights of $W$ at $\sigma \in J$ are non-positive, in which case $N = \operatorname{rank}(W) \cdot |J|,$ i.e., the rank of $H^1_{\text{Iw},g,J,\nabla}(W)$ is $$([K:\QQ_p]-|J|) \operatorname{rank}(W).$$
\end{rem}

\begin{lem} 
	There are canonical isomorphisms 
	$$H^i_{\text{Iw}}(T) \cong \varprojlim_n H^i(K_n,T) \cong\varprojlim_n H^i(K,o_{E}[G(K_n,K)] \otimes T) \cong H^i(G_K,\Lambda \otimes_{o_E}T) $$
\end{lem}
\begin{lem}\label{lem:basechangespectralsequence} 
Let $T \subset V$ be a lattice inside an $E$-linear $p$-adic Galois representation. Let $A$ be a quotient of $\Lambda = o_E\llbracket \Gamma \rrbracket$ then the natural map induces an isomorphism
$$ A \otimes_{\Lambda}^\mathbb{L} \mathbf{R}\Gamma_{cts}(G_K,\Lambda \otimes_{o_E}T) \to \mathbf{R}\Gamma_{cts}(G_K,T\otimes_{o_E} A). $$
\end{lem}
\begin{proof}
This is a special case of \cite[Proposition 1.6.5 (3)]{fukaya2276851formulation}.
\end{proof}
The interesting part of the corresponding spectral sequence is the following: 
{\footnotesize \[\begin{tikzcd}
	{\operatorname{Tor}_3^{\Lambda}(A,H^2_{\text{Iw}}(T))} & {\operatorname{Tor}_2^{\Lambda}(A,H^2_{\text{Iw}}(T))} & {\operatorname{Tor}_1^{\Lambda}(A,H^2_{\text{Iw}}(T))} & {\operatorname{Tor}_0^{\Lambda}(A,H^2_{\text{Iw}}(T))} \\
	{\operatorname{Tor}_3^{\Lambda}(A,H^1_{\text{Iw}}(T))} & {\operatorname{Tor}_2^{\Lambda}(A,H^1_{\text{Iw}}(T))} & {\operatorname{Tor}_1^{\Lambda}(A,H^1_{\text{Iw}}(T))} & {\operatorname{Tor}_0^{\Lambda}(A,H^1_{\text{Iw}}(T))} \\
	&&& {}
	\arrow[from=1-1, to=2-3]
	\arrow[from=1-2, to=2-4]
\end{tikzcd}\]}
The terms above and below are $0$ since $H^i_{\text{Iw}}(T) = 0$ for $i \notin \{1,2\}$ and on the left we have $\operatorname{Tor}_i^{\Lambda}(A,-)= 0$ for $i > d+1$ if $\Gamma$ has no $p$-power torsion since $\Lambda$ has global dimension $d+1.$
From Lemma \ref{lem:basechangespectralsequence} we obtain in low degree terms (compare with \cite[Equation (47)]{venjakob2014kato}): 
\begin{lem} \label{lem:torsequence} Let $\mathfrak{J} \subset \Lambda$ be an ideal.
	Then we have an exact sequence
	\begin{align*}&\operatorname{Tor}^{\Lambda}_1(\Lambda/\mathfrak{J},H^1_{\text{Iw}}(T)) \to H^0(G_K,\Lambda/\mathfrak{J} \otimes_{o_E}T) \to \operatorname{Tor}^{\Lambda}_2(\Lambda/\mathfrak{J},H^2_{\text{Iw}}(T)) \\
	\to 	&\Lambda / \mathfrak{J} \otimes_{\Lambda}H^1_{\text{Iw}}(T) \to H^1(\Lambda/\mathfrak{J} \otimes_{o_E}T) \to \operatorname{Tor}^{\Lambda}_1(\Lambda/\mathfrak{J}, H^2_{\text{Iw}}(T)) \to 0
	\end{align*}
\end{lem}
In ``practice'' $J$ will be the augmentation ideal of an open subgroup, which is given by a regular sequence. 
\begin{prop} \label{prop:torbound} Let $R$ be an $S$-algebra, where $S$ is a noetherian domain.
Let $I \subseteq R$ be an ideal generated by a regular sequence of length $r$ and $M$ an $R$-module which is $S$-finitely generated, then $\operatorname{Tor}_i^{R}(R/I,M)$ is $S$-finitely generated and there exists $C$ depending only on $M$ and $r$ such that $$\operatorname{rank}_{S} \operatorname{Tor}_i^{R}(R/I,M) \leq C$$ for all $i \geq 0.$
\end{prop}
\begin{proof}
Since the Tor groups can be computed using a Koszul complex we see that $\operatorname{Tor}_i(R/I,M)$ is a subquotient of $M^{\oplus \binom{r}{i}}$ and hence has $S$-rank bounded by $\binom{r}{i} \operatorname{rank}_{S}(M).$
\end{proof}

\subsection{Specialising to $H^1_g$}
When $J$ contains all embeddings then the group $H^1_{g,J}$ is, of course, just the usual $H^1_{g}.$
We begin by recalling some of Berger's results in the cyclotomic case (cf. \cite{Berger2005}) which have recently in \cite{Berger2026} been extended to the twisted cyclotomic case.
\begin{thm}(\cite{Berger2005}, \cite{Berger2026})
	Let $V$ be a de Rham $\QQ_p$-linear representation of $G_K$ and $K_\infty/K$ an extension containing an unramified twist of the cyclotomic extension. Let $\operatorname{Fil^1V}$ be the biggest sub representation having all Hodge--Tate weights $\geq 1.$ Then:
	\begin{enumerate} 
\item	$$H^1_{\text{Iw}}(\operatorname{Fil}^1V) \subseteq H^1_{g,\text{Iw}}(V)$$ and the quotient is a finite $\QQ_p$-dimensional vector space. 
	\item If $V$ admits no sub-quotient fixed under the action of $H_K,$ then the inclusion above is an equality.
	\end{enumerate}
\end{thm}
In the situation where we consider $J = \Hom(K,E)$ we can also forget the $E$-linear structure and view all objects just as $\QQ_p$-linear representations to put ourselves in the situation of Berger. 
\begin{lem} \label{lem:sesinvertp}Let $K/\QQ_p$ be finite and suppose that $\Gamma$ is one-dimensional, then 
we have a short exact sequence
$$0
\to 	\Lambda / \mathfrak{J} \otimes_{\Lambda}H^1_{\text{Iw}}(V) \to H^1(\Lambda/\mathfrak{J} \otimes_{o_E}V) \to \operatorname{Tor}^{\Lambda[1/p]}_1(\Lambda/\mathfrak{J}[1/p], H^2_{\text{Iw}}(V)) \to 0$$
\end{lem}
\begin{proof}
We apply the exact functor $\Lambda[1/p] \otimes_{\Lambda} - = \QQ_p \otimes_{\ZZ_p}-$ to the sequence from Lemma \ref{lem:torsequence}. Note that $\operatorname{Tor}$ commutes with flat base change (cf. \cite[\href{https://stacks.math.columbia.edu/tag/00M8}{Tag 00M8}]{stacks-project}). Furthermore, for $\Gamma \cong \ZZ_p \times \Gamma_{tors}$ the ring $\Lambda[1/p]$ is a product of copies of $\ZZ_p\llbracket T\rrbracket[1/p]$ and hence a product of PIDs which forces the $\operatorname{Tor}_2$-terms to vanish. This yields the claim. 
\end{proof}
\begin{lem} \label{lem:injectivefindim} Suppose that $\Gamma$ is one-dimensional, then for any ideal $\mathfrak{J} \subseteq \Lambda$ we have
 $$ \operatorname{dim}_{\QQ_p}\operatorname{Tor}^{\Lambda[1/p]}_1(\Lambda/\mathfrak{J}[1/p], H^2_{\text{Iw}}(V)) \leq \operatorname{dim}_{\QQ_p} H^2_{\text{Iw}}(V) < \infty.$$ 
\end{lem}
\begin{proof}
	The ring $\Lambda[1/p]$ is a product of PIDs. For a module $M$ over a PID $R$ and an ideal $I\subseteq R$ we have that $\operatorname{Tor}^{R}_1(R/I, M)$ can be identified with the $I$-torsion submodule of $M.$ Applying this to each component of $\Lambda[1/p]$ we conclude that the dimension of $\operatorname{Tor}^{\Lambda[1/p]}_1(\Lambda/\mathfrak{J}[1/p], H^2_{\text{Iw}}(V))$ is bounded by the dimension of $H^2_{\text{Iw}}(V).$
\end{proof}
\begin{prop} \label{prop:injectivecokerfinite}Suppose $\Gamma$ is one-dimensional, then the natural map 
	$$ \mathcal{D}_\infty \otimes_\Lambda H^1_{\text{Iw}}(V) \to \varprojlim_n H^1(K_n,V)$$ is injective and the cokernel is finite-dimensional. 
\end{prop}
\begin{proof}
	We have $$\mathcal{D}_\infty \otimes_\Lambda H^1_{\text{Iw}}(V) \cong \varprojlim_n [\mathcal{D}/\mathcal{I}_n \otimes_\mathcal{D} (\mathcal{D} \otimes_{\Lambda} H^1_{\text{Iw}}(V))] \cong \varprojlim_n H^1_{\text{Iw}}(V)/I_n$$
	by Proposition \ref{prop:DmodNablaLimit} using that we have for each level 
	
	$$(\mathcal{D}\otimes_\Lambda H^1_{\text{Iw}}(V))/\mathcal{J}_n \cong \Lambda[1/p]/I_n \otimes_{\Lambda} H^1_{\text{Iw}}(V).$$
	From Lemma \ref{lem:sesinvertp} we have short exact sequences
	$$0
	\to 	\Lambda / I_n \otimes_{\Lambda}H^1_{\text{Iw}}(V) \to H^1(K_n,V) \to \operatorname{Tor}^{\Lambda[1/p]}_1(\Lambda/I_n[1/p], H^2_{\text{Iw}}(V)) \to 0$$
	The kernels are finite-dimensional and hence a Mittag-Leffler system, while the dimension of the right-hand terms is bounded independently of $n.$ Passing to the limit yields the claim. 
\end{proof}

\begin{prop} \label{prop:comparisonberger}
	Let $K_\infty/K$ be the cyclotomic extension, let $V$ be de Rham and let 
	$N_0$ be the number of Hodge--Tate weights of $V$ which are $<1.$
	\begin{enumerate}[(i)]
		\item The natural map 
		 $\mathcal{D}_\infty \otimes_\Lambda H^1_{\text{Iw},g}(V) \to H^1_{\text{Iw},g,J,\nabla}(V)$ has finite-dimensional kernel. 
		 \item 
		 The cokernel is finite-dimensional if and only if $[K:\QQ_p]\dim V - N_0= [K:\QQ_p]\dim\operatorname{Fil}^1V,$
		 if and only if $[K:\QQ_p]\dim\operatorname{Fil}^1V$ is equal to the number of HT weights $\geq 1.$
	\end{enumerate}
\end{prop}
\begin{proof} By restricting to an open subgroup we can without loss of generality assume that $\Lambda[1/p]$ is a PID (isomorphic to $\ZZ_p\llbracket T\rrbracket[1/p]$). 
Since $\mathcal{D}_\infty \otimes_\Lambda H^1_{\text{Iw}}(V) \to \varprojlim_n H^1(K_n,V)$ is injective by Proposition \ref{prop:injectivecokerfinite} it suffices to show that the kernel of 
$$\mathcal{D}_\infty \otimes_\Lambda H^1_{\text{Iw},g}(V) \to \mathcal{D}_\infty \otimes_\Lambda H^1_{\text{Iw}}(V) $$ is finite-dimensional. Since $\Lambda[1/p]$ is a PID we can see that the $$X:=H^1_{\text{Iw}}(V)/H^1_{\text{Iw},g}(V)$$ is a direct sum of a free $\Lambda[1/p]$-module and a finitely generated torsion module. A finitely generated torsion module over $\ZZ_p\llbracket T\rrbracket[1/p]$ is killed by some Weierstraßpolynomial and hence seen to be finite-dimensional. Using that $\mathcal{D}$ is flat over $\Lambda[1/p]$ we obtain a short exact sequence
$$0 \to \mathcal{D}\otimes_\Lambda H^1_{\text{Iw},g}(V) \to \mathcal{D}\otimes_\Lambda H^1_{\text{Iw}}(V) \to \mathcal{D} \otimes_\Lambda X \to 0 $$ and applying 
$\mathcal{D}_\infty \otimes_{\mathcal{D}}-$ to the above sequence we conclude that
$$\operatorname{ker}(\mathcal{D}_\infty \otimes_\Lambda H^1_{\text{Iw},g}(V) \to \mathcal{D}_\infty \otimes_\Lambda H^1_{\text{Iw}}(V) )$$
is a quotient of $\operatorname{Tor}_1^{\mathcal{D}}(\mathcal{D}_\infty,\mathcal{D} \otimes_\Lambda X),$ which is finite-dimensional as required. 
For the second part note that the $\mathcal{D}_\infty$-rank of the left-hand side is equal to $[K:\QQ_p]\dim \operatorname{Fil}^1V$ while the dimension of the right-hand side is $[K:\QQ_p]\dim V-N_0$ by \ref{rem:HTweightsless1}. 
The second if and only if statement is clear. 
\end{proof}
The above proposition shows that our $H^1_{\text{Iw}g,\nabla}(V)$ is a good approximation of $\mathcal{D}_\infty \otimes_\Lambda H^1_{\text{Iw},g}(V)$ when the weights $\geq1$ can all be packaged into a subrepresentation, which is maximal with this property.
We remark that the assumptions of Proposition \ref{prop:comparisonberger} are trivially satisfied if $V$ is de Rham with all Hodge--Tate weights $\geq 1$ or all weights $\leq 0$ but we warn the reader that there exist irreducible representations having arbitrary weights from this it is easy to construct examples where the map from \ref{prop:comparisonberger} has an infinite-dimensional cokernel (cf. \cite{berger2004construction}).

\section{Coadmissibility of $M^{\Psi=1}$} \label{section:analytic}
Throughout this section, $K/\QQ_p$ is a finite Galois extension and $K_\infty/K$
is the Lubin-Tate extension attached to a fixed uniformiser $\pi_K$.
We put $\Gamma_K:=\operatorname{Gal}(K_\infty/K)$ and choose $n\gg0$
such that
\[
 \Gamma_n:=\operatorname{Gal}(K_\infty/K_n)\cong o_K.
\]

\subsection{Preliminaries on analytic $(\varphi,\Gamma)$-modules and their cohomology}
\begin{defn} We say a $B$-pair $W$ is $K$-\textbf{analytic} if it is de Rham of weight $0$ at all $\sigma \neq \id_K.$ 
	We denote $$H^i_{an}(K_n,W):= H^i_{g,\Gal(K/\QQ_p)\setminus \{ \id\}}(K_n,W).$$
\end{defn}
For a complete field extension $F$ of $K$ we define the Robba ring with coefficients in $F$ as  $\cR_F:= \varinjlim_{r \in (0,1)} \cR_F^{[r,1)},$ where $\cR_F^{[r,1)}$ is the ring of Laurent series $\sum_{k\in \ZZ}a_kT^k$ converging on the annulus $r\leq \abs{T} <1$ with coefficients in $F.$
From the Lubin--Tate extension one obtains a one dimensional formal group and we denote by $[a](T)$ the endomorphism corresponding to multiplication by $a \in o_K,$ we also denote $\varphi_K(T):=[\pi_K](T).$ The action on $o_K\llbracket T \rrbracket$ induced by $T \mapsto [a](T)$ can be extended $F$-linearly to the Robba ring $\cR_F.$
\begin{defn}
A $(\varphi_K,\Gamma_K)$-module $M$ over the Robba ring $\cR_F$ is a finitely generated projective $\cR_F$-module  which arises as the base change a finitely generated projective $\cR_K^{[r,1)}$-module $M^{[r,1)}$ (which we call \textbf{model} of $M$) for some $r \gg 0$ \footnote{Note that $r$ needs to be larger than the absolute value of the zeroes of $\varphi_K,$ in order for the map $\varphi_K\colon \cR_K^{[r,1)}\to \cR_K^{[r^{1/q},1)}$ to be well-defined.} with a continuous semi-linear $\Gamma_K$-action and a $\varphi$-semilinear map $$M^{[r,1)} \to M^{[r^{1/q},1)} = \cR_K^{[r^{1/q},1)} \otimes_{\cR_K^{[r,1)}} M^{[r,1)}$$ which commutes with the action of $\Gamma_K$ and such that the linearised map
$$M^{[r^{1/q},1)} \otimes_{\cR_K^{[r,1)},\varphi_K}M^{[r,1)} \to M^{[r^{1/q},1)}$$ is an isomorphism. 
We say $M$ is $K$-analytic if the $\Gamma_K$-action is $K$-analytic with respect to the $K$-analytic group structure obtained from the chart $$\Gamma_K \xrightarrow{\chi_{\mathrm{LT}}} o_K^\times,$$ i.e., if the induced derived action 
$\operatorname{Lie}(\Gamma_K) \times M \to M$  is $K$-bilinear. 
\end{defn}
If $M$ is a $K$-analytic $(\varphi_K,\Gamma_K)$-module over $\cR_F,$ then the action of $\Gamma_K$ extends uniquely by $K$-linearity to an action of the $K$-analytic distribution algebra $\mathcal{D}^{K}(\Gamma_K,F),$ which is a quotient of the algebra $\mathcal{D}^{\QQ_p}(\Gamma_K,F).$
We denote by $\Psi$-the canonical left-inverse operator of $\varphi_M.$ It coincides with the $\Psi$-operator used in \cite{BFFanalytic} and differs by a factor of $\frac{q}{\pi}$ from the operator $\psi_L$ from \cite{SchneiderVenjakobRegulatorBook}.
\begin{defn}
We define the analytic cohomology $H^i_{an}(\Psi \times \Gamma_n,M)$ of $M$ as the cohomology of the total complex $C_{an}(\Gamma_n,M)$ of the double complex
$$\mathbf{R}\Hom_{\mathcal{D}^{K}(\Gamma_n,K)}(K,M) \xrightarrow{\Psi-1} \mathbf{R}\Hom_{\mathcal{D}^{K}(\Gamma_n,K)}(K,M).$$
\end{defn}
In the following, 
all completed tensor products are Hausdorff completed projective tensor products. Recall that the projective tensor product (\cite[Remark 10.3.3]{PGS}) represents jointly continuous bilinear maps. It is easy to see that the Hausdorff completion of the projective tensor product represents jointly continuous bilinear maps valued in complete Hausdorff spaces. If $A$ is a topological $F$-algebra, we denote by $-\widehat{\otimes}_A-$ the Hausdorff completion of the algebraic tensor product $-\otimes_A-$ with respect to the quotient topology from $-\otimes_F-.$ If $X,Y$ are two topological $A$-modules, then a (jointly) continuous $A$-bilinear map valued in a complete Hausdorff $A$-module factors uniquely over $X \widehat{\otimes}_AY.$
\begin{rem}
If $M$ is defined over $\cR_F$ for some complete field extension it is easy to check that 
$$\mathbf{R}\Hom_{\mathcal{D}^{K}(\Gamma_K,K)}(K,M) \cong \mathbf{R}\Hom_{\mathcal{D}^{K}(\Gamma_K,F)}(F,M).$$ Indeed by \cite[Theorem 6.5]{Kohlhaase} one can find a resolution of $K$ by finite free $\mathcal{D}^K(\Gamma_K,K)$-modules and the completed base change of said resolution is a projective resolution of $F$ as a $F\hat{\otimes}_K\mathcal{D}(\Gamma_K,K)=\mathcal{D}(\Gamma_K,F)$-module.
\end{rem}
\begin{prop}
The following hold:
\begin{enumerate}
\item The category of $K$-analytic $K$-$B$-pairs is equivalent to the category of $K$-analytic $(\varphi_K,\Gamma_K)$-modules. 
\item Under the above equivalence we have isomorphisms 
 $$H^i_{an}(K_n,W) \cong H^i_{an}(\Psi \times \Gamma_n,M(W)),$$ where $M(W)$ denotes the $(\varphi_K,\Gamma_K)$-module attached to a $K$-analytic $B$-pair $W.$
 \item $H^i_{an}(\Psi\times \Gamma_n,M)$ is isomorphic to the $\Gamma_n$-invariants of the cohomology of the complex 
 $$0 \to M \xrightarrow{\Psi-1,\nabla} M^2 \xrightarrow{\nabla \circ p_1 - (\Psi-1)\circ p_2}M \to 0,$$ where $\nabla \in \operatorname{Lie}(\Gamma_n) \cong K$ is a generator and $p_i \colon M^2 \to M $ are the projections. 
\end{enumerate}
\end{prop}

\begin{proof}
For the first point see \cite[Theorem A]{Poyeton2023-sw}.
The second point is part of \cite[Theorem 1]{steingart2025partial}.
The third point is a consequence of \cite[Theorem 4.10]{Kohlhaase}. Note that the abelian group $\Gamma_K$ satisfies the required hypotheses by Theorem 6.5 \textit{loc.cit.}. 
\end{proof}
In \cite{STEINGART202424} the analytic cohomology of $K$-analytic $(\varphi_K,\Gamma_K)$-modules over affinoid families is studied. For an affinoid algebra $A$ over $F$ one can analogously define $(\varphi_K,\Gamma_K)$-modules over the relative Robba ring $A \hat{\otimes}_F \cR_F$ and their analytic cohomology.

\subsection{Consistent complexes and Deformations}

\begin{defn}
Let $(A_n)_{n \in \NN}$ be a countable projective system of noetherian rings and $A:= \varprojlim_n A_n.$
Let $\mathcal{C}$ be the abelian category whose objects are projective systems of abelian groups $M_n,$ such that each $M_n$ is an $A_n$-module and the transition map $M_{n+1} \to M_n$ is $A_{n+1}$-linear. 
Given an object $C$ in the derived category $D(\mathcal{C})$ we can obtain a family of complexes $C_n \in D(A_n)$ in an obvious way. 
We say $C$ is \textbf{consistent} if 
\begin{enumerate}
\item $C$ is bounded below.
\item $A_n \otimes_{A_{n+1}}^\mathbb{L} C_{n+1} \to C_n$ is an isomorphism in $D(A_n)$ for every $n.$
\item The cohomology groups of $C_n$ are finitely generated over $A_n$ for every $n.$
\end{enumerate}
\end{defn}
\begin{lem} \label{lem:consistent} Let $A = \varprojlim_n A_n$ be a Fréchet-Stein algebra, where each $A_n$ is Banach and $A_{n+1} \to A_n$ is flat with dense image. Then the functors
	$$C \mapsto (A_n \otimes_A C)_n$$
	and 
	$$C_n \mapsto \mathbf{R}\varprojlim C_n$$ define an equivalence between the full subcategory $D^+_c(A)$ of $D^+(A)$ consisting of bounded below complexes with coadmissible cohomology and the full subcategory of consistent complexes in $D^+(\mathcal{C}).$

\end{lem}
\begin{proof}
This is \cite[Proposition 3.15]{steingart2024iwasawa}.
\end{proof}
\begin{prop}\label{prop:DFMconsistent}
Let $M$ be a $K$-analytic $(\varphi_K,\Gamma_K)$-module over $\cR_{\CC_p}.$ Write $D:=\mathcal{D}^K(\Gamma_n,\CC_p) \cong \varprojlim D_\nu$ and equip 
$$\mathbf{Dfm}_\nu(M):=M \hat{\otimes}_K D_\nu$$ with $\Gamma$-action induced by
$\gamma(m \otimes \lambda):= \gamma m \otimes [\gamma^{-1}]\lambda.$
Then the projective system $(C_{an}(\mathbf{Dfm}_\nu(M)))_\nu$ defines a consistent complex of $D$-modules. 
\end{prop}
\begin{proof}
Recall that the distribution algebra $D^K(\Gamma_n,\CC_p)$ is isomorphic to the convergent power series on the open unit disc in some coordinate $Z,$ hence $[D \xrightarrow{Z}D]$ is a projective resolution of the $\CC_p$ as a $D$-module. The result is now just a translation of \cite[Remark 3.16]{steingart2024iwasawa} using that the complex $C_{\Psi,Z}(\mathbf{Dfm}_n(M))$ used in \textit{loc.cit.} is quasi-isomorphic to our $C_{an}(\mathbf{Dfm}_n(M)).$
\end{proof}

\subsection{Coadmissibility and Descent}

\begin{lem}\label{lem:fixed-vectors}
Let 
\[
V_0\xrightarrow{\iota_0}V_1 \xrightarrow{\iota_1} \dots V_j \xrightarrow{\iota_j} V_{j+1}\dots
\]
be a direct system of abelian groups. Suppose that each $V_j$ comes with an endomorphism $T_j\colon V_j \to V_j$
and suppose that there exist maps $\rho_{j+1}\colon V_{j+1} \to V_j$ such that 
\[
\rho_{j+1} \circ \iota_j = T_j, \qquad \iota_j\circ \rho_{j+1} =T_{j+1},
\]
which implies
\[ 
\iota_j \circ T_{j} = T_{j+1} \circ \iota_j. 
\]
Let $V=\varinjlim_j V_j$ and denote by $T\colon V\to V$ the induced endomorphism. Then for all $j\ge 0$, the canonical map 
\[
V_j^{T_j=1} \to V^{T=1}
\]
is an isomorphism.
\end{lem}
\begin{proof}
Let $x\in V^{T=1} =\left( \varinjlim_j V_j\right)$. Represent $x$ by some $v\in V_m$. Assume 
$m\ge j$. Since $T(x)=x$, there exists some $k\ge m$ such that \[
v_k:=\iota_{m,k}(v)\in V_k
\]
and $T_k(v_k)=v_k$. Let
\[
\rho_{k,j}:= \rho_{j+1} \circ \dots \circ \rho_k \colon V_k\to V_j 
\]
and put $w:=\rho_{k,j}(v_k)$. Then \[
\rho_{k,j} \circ T_k = T_j \circ \rho_{k,j},
\]
so that $T_j(w) = \rho_{k,j}(v_k)=w$, whence \[
\iota_{j,k}(w)=T_k^{k-j}(v_k)=v_k.
\]
This shows that $w\in V_j^{T_j=1}$ represents $x$ and shows surjectivity.
On the other hand, suppose that $w\in V_j^{T_j=1}$ is mapped to $0$
under $\iota_j\colon V_j\to V$. Then, we find $k\ge j$ such that $\iota_{j,k}(w)=0$. Applying $\rho_{k,j}$ we then find \[
0=\rho_{k,j}\circ \iota_{j,k} (w)= T_j^{k-j}(w)=w,
\]
which shows injectivity.
\end{proof}
\begin{lem}
For every $j\ge n(M)$, put \[
E_{v,j}:= D_v \hat{\otimes}_{\CC_p,\pi} M^{(j)},\qquad Q_{v,j}:= D_v\hat{\otimes}_D M^{(j)}.
\]
Then the natural map \[
\theta_j\colon M^{(j)}\to \varprojlim_v Q_{v,j}
\]
is a $D$-linear, $\Psi$-equivariant topological isomorphism.
\end{lem}
\begin{proof}
This is contained in \cite[Lemma 3.21.]{steingart2024iwasawa}. By this lemma, for every $v$ there is a compatible strict short exact sequence 
of Fréchet spaces \[
0 \to E_{v,j} \xrightarrow{Z_{diag}} E_{v,j}\to Q_{v,j}\to 0
\]
Since the transition maps $E_{v+1,j}\to E_{v,j}$ are dense, by Mittag Leffler, this gives a strict short exact sequence \[
0 \to \varprojlim_v E_{v,j}\xrightarrow{Z_{diag}} \varprojlim_{v} E_{v,j} \to \varprojlim_v Q_{v,j}\to 0.
\]
Since projective limits commute with projective tensor products $\varprojlim_v E_{v,j} = D\hat{\otimes}_{\CC_p,\pi} M^{(j)}$.
Again by \cite[Lemma 3.21.]{steingart2024iwasawa}, we have a strict exact sequence 
\[
0\to  D\hat{\otimes}_{\CC_p,\pi} M^{(j)}\xrightarrow{Z_{diag}}  D\hat{\otimes}_{\CC_p,\pi} M^{(j)} \to M^{(j)} \to 0.
\]
Combining with the above proves the lemma.
\end{proof}
\begin{thm} \label{thm:psicoadmissible}
Let $M$ be a $K$-analytic $(\varphi_K,\Gamma_K)$-module over $\cR_{\CC_p}$. Then $M^{\Psi=1}$ is a coadmissible $D^K(\Gamma_n,\CC_p)$-module.
\end{thm}
\begin{proof}
For every $v$ set\[
A_v:= C_{an}(\operatorname{Dfm}_v(M))
\]
By Proposition~\ref{prop:DFMconsistent}, the system $(A_v)_v$ is consistent. It follows that 
\[
Y:= H^1(R\varprojlim_v A_v)
\]
is a coadmissible $D$-module and that the natural map \[
Y\xrightarrow{\sim} \varprojlim_v H^1(A_v)
\]
is an isomorphism. Put\[
Q_v:= \varinjlim_j Q_{v,j}. 
\]\footnote{ $Q_v$ is denoted by $D_v\tilde{\otimes}_D M$ in \cite{steingart2024iwasawa} .}
In the proof of \cite[Theorem 3.24]{steingart2024iwasawa} it was shown that there are compatible quasi-isomorphisms
\[
A_v \xrightarrow{\sim} C_{\Psi}(Q_v).
\]
Note that this part of the proof in \emph{loc. cit.} does not make use of the coadmissible assumption. Together with the isomorphism above 
we obtain a $D$-linear isomorphism 
\[
Y \xrightarrow{\sim} \varprojlim_v Q_v^{\Psi=1}.
\]
We now choose $j_0\ge n(M)$ large enough such that $\iota_{j,j+1}\colon M^{(j)}\to M^{(j+1)}$
and $\rho_{j+1,j}:= \Psi\colon M^{(j+1)}\to M^{(j)}$ and $\Psi:= \rho_{j+1,j} \circ \iota_{j,j+1}$ 
are all defined for all $j\ge j_0$, so that in particular $M=\varinjlim_{j\ge j_0}M^{(j)}$ and similarly 
for $Q_v$. Then, by the lemma above we have 
\[
Q_{v,j_0}^{\Psi=1} = Q_v^{\Psi=1}, \qquad (M^{(j_0)})^{\Psi=1} = M^{\Psi=1}.
\]
Combining everything, we obtain 
\begin{align*}
Y &\cong \varprojlim_v Q_v^{\Psi=1}\\
&\cong \varprojlim Q_{v,j}^{\Psi=1}\\
&\cong \left(\varprojlim_v Q_{v,j}\right)^{\Psi=1}\\
&\cong \left(M^{(j)}\right)^{\Psi=1}\\
& \cong M^{\Psi=1}.
\end{align*}
One checks that all the involved maps are $D$-linear.
\end{proof}

In a next step we descend our results from $\CC_p$ to $K$. 
For now let $K\subseteq F\subset \CC_p$ be a complete intermediate field. 
Our descent result can be seen as a topological version of classical faithfully flat descent for modules in the setting 
of Fréchet--Stein algebras. Attempting a Galois descent argument in this context runs into trouble, for example the trivial (semi-linear) representation $\CC_p$ already has non-trivial $H^1(G_{\QQ_p},\CC_p).$ Instead we will work with an abstract section $\CC_p \to K$ (obtained from the Hahn--Banach Theorem) and some functional analysis to arrive at a ``faithful flatness'' result for completed base change which will be sufficient for our specific situation.  Let $\mathcal{C}$  
denote the category of complete Hausdorff locally convex $K$-vector spaces with continuous $K$-linear maps as morphisms.
The key input is the following straightforward consequence of the Hahn--Banach theorem.
\begin{lem}\label{lem:split-descend}
Let $E, E'\in \mathcal{C}$.
\begin{enumerate}
	\item $E=0$ if and only if $F\hat{\otimes}_K E =0$;
	\item Let $f\colon E\to E'$ be a morphism in $\mathcal{C}$ and put $f_F:=\operatorname{id}_F\hat{\otimes}f$. If 
	$f_F$ is injective (resp. surjective) (resp. bijective) (resp. a homeomorphism), then so is $f$.
\end{enumerate}
\end{lem}
\begin{proof}
(1) follows from (2). By the Hahn-Banach theorem the identity on the subspace 
$K\cdot 1\subseteq F$ can be extended to a bounded $K$-linear map  \[
\ell\colon F\to K,\qquad \ell(1)=1.
\]
For every complete Hausdorff locally convex $K$-space $E$, we therefore have natural continuous maps \[
\iota_E\colon E \to F\hat{\otimes}_K E, \qquad e\mapsto 1\otimes e,
\]
and \[
\pi_E:=\ell\hat{\otimes} \operatorname{id}_E\colon F\hat{\otimes}_K E\to E,
\]
satisfying \begin{equation}\label{eq:pi-iota-id}
\pi_E\circ \iota_E = \operatorname{id}_E.
\end{equation} 
In particular, $\iota_E$ identifies $E$ with a closed complemented subspace in $F\hat{\otimes}_K E$.
Suppose now that $f\colon E\to E'$ is continuous. We have the following relations \[
f_F\circ \iota_E = \iota_{E'} f, \qquad \pi_{E'}f_F = f \pi_E.
\]
For injectivity, if $f(x)=0$, then \[
f_F(\iota_E(x))=\iota_{E'}(f(x))=0.
\]
Injectivity of $f_F$ gives $\iota_E(x)=0$ and applying $\pi_E$ gives $x=0$. For surjectivity, given $y\in E'$, choose \[
z\in F\hat{\otimes}_K E \text{ with } f_F(z)=\iota_{E'}(y).
\]
Again using the relations above, we obtain \[
f(\pi_E(z))=\pi_{E'}(f_F(z))=\pi_{E'}(\iota_{E'}(y))=y.
\]
Finally if $f_F$ is a topological isomorphism, its inverse is continuous and \[
f^{-1} = \pi_E \circ f_F^{-1}\iota_{E'}
\]
is a continuous inverse to $f$.
\end{proof}

Let $A=\varprojlim A_m$ be a Fréchet--Stein algebra over $K$ and define $B_m:=F\hat{\otimes}_K A_m$ and $B=\varprojlim B_m$.
Assume that $B$ is a Fréchet--Stein algebra over $F$.
\begin{thm}\label{thm:galoisdescentadmissible}
Let $N$ be a Fréchet space over $K$ and suppose that \begin{enumerate}
	\item $N$ admits a continuous action of $A$;
	\item $P:=F\hat{\otimes}_K N$ is a coadmissible $B$-module, whose given Fréchet topology coincides with its coadmissible topology.
\end{enumerate}
Then $N$ is a coadmissible $A$-module.
\end{thm}
\begin{proof}
We use the notation from the proof of Lemma~\ref{lem:split-descend}. We define $N_m:=A_m\hat{\otimes}_A N$. 
Write $P=\varprojlim P_m$ with $P_m$ being finitely generated Banach $B_m$-modules. We claim that 
\begin{equation}\label{eq:distributive-complexted-tp}
F\hat{\otimes}_K N_m \cong B_m \hat{\otimes}_B P \cong B_m\otimes_B P \cong P_m.
\end{equation}
The second and third isomorphisms are clear by coadmissibility of $P$, hence it suffices to show that\[
\alpha\colon F\hat{\otimes}_K (A_m\hat{\otimes}_A N) \to (F\hat{\otimes}_K A_m)\hat{\otimes}_{F\hat{\otimes}_K A}(F\hat{\otimes}_K N)
\]
given on elementary tensors by \[
\lambda\otimes (a\otimes x) \mapsto (\lambda\otimes a) \otimes (1\otimes x)
\]
is a topological isomorphism. Its inverse is induced by
$$((\lambda\otimes a),(\mu\otimes x))\longmapsto\lambda\mu\otimes(a\otimes x).$$
The universal property of the completed projective tensor product, followed by the balanced quotient, gives continuous maps in both directions. 
They are inverse on elementary tensors and hence everywhere by density and completeness.

We claim that $N_m$ is finite over $A_m$ for which we follow the argument of \cite[Lemma 2.1.12]{KPX}. By \eqref{eq:distributive-complexted-tp} and \eqref{eq:pi-iota-id}, 
we can identify $N_m$ with a closed complemented $K$-subspace of the Banach space $P_m$. Thus $N_m$ is itself Banach.
Choose $B_m$-module generators $y_1,\dots, y_t\in P_m$ of $P_m$. The continuous surjection \[
q\colon B_m^t\to P_m,\qquad (b_i)\mapsto \sum_i b_iy_i
\]
is open by open mapping for Banach spaces. We define \[
U:=\{(b_j)\in B_m^t\colon \operatorname{max}_j \norm{b_j}_{B_m} < \frac{1}{2}\}.
\]
Since $q$ is open, $q(U)\subset P_m$ is open. In particular, we find $\delta>0$ such that
\begin{equation}\label{eq:small-enough}
\{z\in P_m\colon \norm{z}_{P_m} <\delta\} \subset q(U).
\end{equation}
By \eqref{eq:distributive-complexted-tp}, we have the topological identification $P_m\cong F\hat{\otimes}_K N_m$, and by definition,
$F\otimes_K N_m$ is a dense subspace. Hence, for each $i$, we can choose \[
y_i' \in F\otimes_K N_m \text{ such that } \norm{y_i'-y_i}_{P_m} <\delta.
\]
By \eqref{eq:small-enough}, applied to $z=y_i'-y_i$ we find that 
\begin{equation}\label{eq:new-generators}
y_i'-y_i = \sum_j a_{ij} y_j,\qquad \norm{a_{ij}}_{B_m} < \frac{1}{2} < 1.
\end{equation}
Put $C:=(a_{ij})_{ij}$. Since $\norm{C} = \max_{i,j} \norm{a_{ij}}_{B_m} < 1$, the series \[
(I+C)^{-1} = I-C+C^2-C^3+\dots
\]
converges in $M_t(B_m)$. Hence the $y_i'$ generate $P_m$. We can find finitely many $x_1,\dots x_s\in N_m$ such that\[
y_i' =\sum_j \lambda_{ij} \otimes x_j.
\]
Therefore, the $1\otimes x_1,\dots, 1\otimes x_s$ generate $F\hat{\otimes}_K N_m$. Consider\[
f\colon A_m^s\to N_m,\qquad (a_j) \mapsto \sum_j a_jx_j.
\]
Then $f_F=\operatorname{id}_F\hat{\otimes}_K f$ is surjective, hence by Lemma~\ref{lem:split-descend}, $f$ is surjective as well.
Thus $N_m$ is finitely generated over $A_m$. Next we claim that $(N_m)_{m\ge1}$ defines a coherent system over $A$. 
The natural transition maps induce maps
\begin{equation}\label{eq:transition-maps-blabla}
	f_m\colon A_m\otimes_{A_{m+1}}N_{m+1} \to N_m,
\end{equation}
with the source and target being finitely generated $A_m$-modules. We claim that
\begin{equation}\label{eq:transition-maps-completed}
A_m\otimes_{A_{m+1}} N_{m+1} \cong A_m\hat{\otimes}_{A_{m+1}} N_{m+1}.
\end{equation}
Indeed, since $A_{m+1}$ is noetherian, we find a finite presentation\[
A_{m+1}^r \to A_{m+1}^s \to N_{m+1} \to 0.
\]
Base change to $A_m$ shows that $A_m \otimes_{A_{m+1}} N_{m+1}$ is the quotient of a finite free $A_m$ module by a closed submodule.
Hence $A_m \otimes_{A_{m+1}} N_{m+1}$ with its quotient topology is already complete and it suffices to show that this coincides with the 
projective tensor topology. To verify this, write $R=A_{m+1}$, $S=A_m$ and $Q=N_{m+1}$, 
and let $T=S\otimes_RQ$ carry the quotient Banach topology from the finite presentation. 
The balanced map $S\times Q\to T$ is jointly continuous: lift it to $S\times R^s$ 
and descend along the open quotient $R^s\to Q$. It induces $S\widehat{\otimes}_RQ\to T$. 
Conversely, the chosen generators of $Q$ induce a continuous map $S^s\to S\widehat{\otimes}_RQ$ vanishing on the relation submodule, 
hence a map $T\to S\widehat{\otimes}_RQ$. The maps are inverse on elementary tensors, and then on their completions.
By \eqref{eq:distributive-complexted-tp}, the completed base change of \eqref{eq:transition-maps-blabla} from $K$ to $F$
identifies with \[
B_m\otimes_{B_{m+1}} P_{m+1} \to P_m.
\]
Since $P$ is coadmissible, this is an isomorphism, so that $f_m$ is already an isomorphism by Lemma~\ref{lem:split-descend}.
It remains to identify $N$ with $\varprojlim N_m$. Let \[
u\colon N\to \varprojlim_m N_m
\]
be the canonical map. Using \eqref{eq:distributive-complexted-tp} together with the maps in the proof of Lemma~\ref{lem:split-descend}
we obtain inclusion and retraction maps \[
\iota_m\colon N_m \to P_m,\qquad  \pi_m \colon P_m\to N_m,
\]
Furthermore, let \[
u_m\colon N\to N_m,\qquad v_m\colon P\to P_m
\]
be the natural maps and let $\iota_N\colon N\to P$ and $\pi_N\colon P\to N$
be the inclusion and retraction maps from the proof of Lemma~\ref{lem:split-descend}.
One easily verifies the following identities:
\begin{equation}\label{eq:some-identities}
	v_m\circ \iota_N = \iota_m \circ u_m,\qquad \pi_m\circ v_m = u_m \circ \pi_N.
\end{equation}
We directly show that $u$ is an isomorphism. For injectivity, suppose that $u(x)=0$. Then,
$u_m(x)=0$ for all $m\ge 1$. By \eqref{eq:some-identities}, we have\[
v_m(\iota_N(x)) = \iota_m(u_m(x))= 0\qquad {\text{ for all }} m.
\]
Since $P\cong \varprojlim_m P_m$ via $y\mapsto (v_m(y))_m$, we see that $\iota_N(x)=0$.
Applying the retraction $\pi_N$ we obtain \[
x=\pi_N(\iota_N(x)) = 0,
\]
proving injectivity. For surjectivity, let $(x_m)_m \in \varprojlim_m N_m$. Then $(\iota_m(x_m))_m \in \varprojlim_m P_m = P$. 
So, there exists some $y\in P$, such that \[
v_m(y)= \iota_m(x_m) \qquad \text{for all } m.
\]
Set $x=\pi_N(y)$. Then by \eqref{eq:some-identities}\[
u_m(x)=\pi_m v_m(y)=\pi_m\iota_m (x_m) = x_m.
\]
Thus $u(x)=(x_m)_m$ proving surjectivity.
Since $u$ is a continuous bijection of Fréchet spaces, $u$ is a topological isomorphism, finishing the proof.
\end{proof}

\begin{lem}\label{lem:coadm-top-given-top}
Let $M$ be a Fréchet space over $F$ and let $A$ be a Fréchet-Algebra over $F$. Suppose that 
$M$ admits a separately continuous action of $A$, which makes it a coadmissible, finitely generated $A$-module.
Then, the coadmissible topology of $M$ coincides with its given Fréchet topology.
\end{lem}
\begin{proof}
Take any surjection $A^d \to M.$ This map is a continuous surjection of Fréchet spaces and hence open by the Open Mapping Theorem. 
\end{proof}

\subsection{Computing the rank of $M^{\Psi=1}$}
Having shown that $M^{\Psi=1}$ is coadmissible over $D^K(\Gamma_n,\CC_p)$, we now turn to investigating its finiteness and rank. Let $F$ be either $\CC_p$ or $K$.
For our choice of $\Gamma_n$ we have that $D=\mathcal{D}^K(\Gamma_n,F)$ is a Prüfer domain, which
can be written as a projective limit of Dedekind domains $D_n$ with flat transition maps each having dense image with respect to the 
corresponding Banach norms. 
\begin{defn}
Let $M$ be a $K$-analytic $(\varphi_K,\Gamma_K)$-module over $\cR_F$, we define the heart of $M$ as $$\mathcal{C}(M):= (1-\varphi_K)M^{\Psi=1}.$$
\end{defn}
\begin{lem}
Let $M$ be a $K$-analytic $(\varphi_K,\Gamma_K)$-module over $\cR_F$.
There is a short exact sequence
$$0 \to M^{\varphi_K=1} \to M^{\Psi=1} \to \mathcal{C}(M) \to 0.$$
The heart of $M$ is a flat module over $D$ and $M^{\varphi_K=1}$ is finite dimensional. 
\end{lem}
\begin{proof}
Exactness is clear by definition.
Consider first the case $F=\CC_p$.
The heart is torsion free since it is a submodule of 
$M^{\Psi=0}$ which is projective as a module over the Robba ring of $\Gamma_n$ \cite[Theorem 2.4.5.]{STEINGART202424}, in particular torsion free as a module over the subring $D.$ 
A torsion free module over a Prüfer domain is flat. Finite dimensionality of $M^{\varphi_K=1}$ holds by \cite[Remark 3.8]{steingart2024iwasawa}.
Now consider the case $F=K$. The heart of $M$ is a submodule of the heart of $M_{\CC_p}=\CC_p\widehat{\otimes}_KM$ and the latter is torsion free over $D^K(\Gamma_n,K)\subset D^K(\Gamma_n,\CC_p)$,
hence the former is torsion free as well. For the finite dimensionality of $M^{\varphi_K=1}$ consider the injective map \[
\CC_p\otimes_K \cR_K \to \cR_{\CC_p}.
\]
Since $M$ is finite free over $\cR_K$, there is also an injective map $\CC_p\otimes_K M \to M_{\CC_p}$. Tensoring the inclusion $M^{\varphi_K=1} \subset M$ with $\CC_p$ over $K$ 
gives an injective map $\CC_p\otimes_K M^{\varphi_K=1} \to M_{\CC_p}$, which factors through $M_{\CC_p}^{\varphi_K=1}$ and finite dimensionality of 
$M^{\varphi_K=1}$ follows from finite dimensionality of $M_{\CC_p}^{\varphi_K=1}$.
\end{proof}
\begin{lem} \label{lem:finitedimcoad}
Let $V$ be a finite dimensional $F$-vector space equipped with a continuous (with respect to the Hausdorff topology) action of $D,$ then $V$ is finitely presented, in particular, coadmissible.
\end{lem}
\begin{proof}
Take any $D$-linear surjection $D^d \to V.$ Since $V$ is Hausdorff the kernel is closed and by \cite[Lemma 1.9]{Berger} a closed submodule of a finitely generated projective $D$-module is finitely generated projective. Finitely presented $D$-modules are coadmissible by \cite[Corollary 3.4]{schneider2003algebras}.
\end{proof}
\begin{lem} \label{lem:finitepres} Let $M$ be a $K$-analytic $(\varphi_K,\Gamma_K)$-module over $\cR_{\CC_p}$. The heart
$\mathcal{C}(M)$ is finitely generated projective as a module over $D$ and $M^{\Psi=1}$ is finitely presented of the same rank as $\mathcal{C}(M).$
\end{lem}
\begin{proof}
By Theorem \ref{thm:psicoadmissible} and Lemma \ref{lem:finitedimcoad} we have that $\mathcal{C}(M)$ is coadmissible. 
For every $n$ the module  $D_n \otimes_D \mathcal{C}(M)$ is finitely generated and flat and since $D_n$ is a Dedekind domain (even a PID), 
we conclude that it is finite projective and because $$D_{n}\otimes_{D_{n+1}}(D_{n+1} \otimes_D\mathcal{C}(M)) = D_n \otimes_D \mathcal{C}(M)$$ they all have the same rank which we denote by $r.$
For a finite projective module $P$ over a Dedekind domain $R$ we have that $P\cong R^{r-1}\oplus I$ for some invertible ideal $I\subset R$. In particular, $P$ is generated by at most 
$r+1$ elements. Hence, $D_n \otimes_D \mathcal{C}(M)$ is generated by at most $r+1$ elements for all $n$.
The proof of \cite[Proposition 2.1.13.]{KPX} shows that $\mathcal{C}(M)$ can be generated by $2(r+1)$ global sections. Hence $\mathcal{C}(M)$ is a finitely generated flat module over the Prüfer domain $D$ which implies it is finitely generated projective hence finitely presented. 
$M^{\Psi=1}$ is thus an extension of finitely presented modules and by \cite[\href{https://stacks.math.columbia.edu/tag/0519}{Tag 0519}]{stacks-project} finitely presented. 
The statement about the rank follows from $M^{\varphi_K=1}$ being finite dimensional and hence torsion. 
\end{proof}

We can now descend coadmissibility and finiteness of $M^{\Psi=1}$ from $\CC_p$ to $K$.
Recall from the proof of \cite[Propostion~A.2]{SchneiderVenjakobRegulatorBook} that $M^r/(\Psi-1)$ is finite-dimensional for $r\gg 0$, hence $\Psi-1\colon M^r\to M^r$ is a strict 
map of Fréchet spaces by \cite[Lemma 1.11]{steingart2024iwasawa} and by Lemma~1.12 in loc.~cit. we have 
\[
(M^r)^{\Psi=1}\widehat{\otimes}_K\CC_p \cong (M^r_{\CC_p})^{\Psi=1}.
\]

\begin{cor} \label{cor:perfectpsi}
Let $M$ be a $K$-analytic $(\varphi,\Gamma_K)$-module over $\cR_K$. Then $M^{\Psi=1}$ is a coadmissible $D^K(\Gamma_n,K)$-module. 
The heart $\mathcal{C}(M)$ is finitely generated projective as a module over $D^K(\Gamma_n,K)$ and $M^{\Psi=1}$ is finitely presented of the same rank as $\mathcal{C}(M).$
\end{cor}
\begin{proof}
We equip $X:=M^{\Psi=1}$ with the Fréchet topology of some $M^r$ for $r$ close enough to $1.$ 
By Lemma \ref{lem:coadm-top-given-top} together with Lemma \ref{lem:finitepres} the base change $(M^r_{\CC_p})^{\Psi=1}$ carries its coadmissible topology.
In particular, $X$ satisfies the hypothesis of Theorem \ref{thm:galoisdescentadmissible} and is thus coadmissible. The remaining 
statements can now be proven by the same argument as in Lemma \ref{lem:finitepres}.
\end{proof}

\begin{lem} \label{lem:anbasechange}
Let $M$ be a $K$-analytic $(\varphi_K,\Gamma_K)$-module over $\cR_{K}$ and $M_{\CC_p} = \CC_p\widehat{\otimes}_KM$ then 
$H^i_{an}(M_{\CC_p}) = \CC_p \otimes_K H^i_{an}(M),$ where on the right hand side we use $\Psi$ to define the analytic cohomology. 
\end{lem}
\begin{proof}
We recall that for $X =M,M_{\CC_p}$ the cohomology $H^i_{an}(X)$ can be computed as $H^0(\Gamma_n,H^i_{\Psi,\nabla}(X)).$ From the double complex spectral sequence for the corresponding Koszul complexes we obtain a commutative diagram with exact rows  
{\footnotesize \[\begin{tikzcd}
	0 & {\mathbb{C}_p \otimes_{K}M^{\Psi=1}/\nabla} & {\mathbb{C}_p \otimes_{K}H^1_{\Psi,\nabla}(M)} & {\mathbb{C}_p\otimes_K [M/(\Psi-1)]^{\nabla=0}} & 0 \\
	0 & {M_{\mathbb{C}_p}^{\Psi=1}/\nabla} & {H^1_{\Psi,\nabla}(M_{\mathbb{C}_p})} & {[M_{\mathbb{C}_p}/(\Psi-1)]^{\nabla=0}} & 0
	\arrow[from=1-1, to=1-2]
	\arrow[from=1-2, to=1-3]
	\arrow[from=1-2, to=2-2]
	\arrow[from=1-3, to=1-4]
	\arrow[from=1-3, to=2-3]
	\arrow[from=1-4, to=1-5]
	\arrow[from=1-4, to=2-4]
	\arrow[from=2-1, to=2-2]
	\arrow[from=2-2, to=2-3]
	\arrow[from=2-3, to=2-4]
	\arrow[from=2-4, to=2-5]
\end{tikzcd}\]
}
Because the action of $\nabla$ is trivial and $K$ is a projective $D/\nabla$-module, applying $H^0(\Gamma_n,-)$ to the diagram preserves the exactness of the rows. 
By \cite[Lemma 1.12]{steingart2024iwasawa} short exact sequences of strict maps of Fréchet spaces remain exact after completed base change with a Banach space of countable type (note that $\mathbb{C}_p$ contains a $\overline{K}$ as a countable dimensional dense subspace and is hence of countable type over $K.$).
The right-most vertical arrow is an isomorphism because $\Psi-1$ is a strict map and $M/(\Psi-1)$ is finite dimensional and hence already complete. 
For the left most map we make use of the coadmissibility of $M^{\Psi=1}.$ Because $\Gamma_n$ is commutative, the map $\nabla\colon M^{\Psi=1} \to M^{\Psi=1}$ is a $D$-linear map between coadmissible modules and hence strict. 
We hence have $\CC_p\widehat{\otimes}_K(M^{\Psi=1}/\nabla) \cong (M_{\CC_p}^{\Psi=1})/\nabla.$ The $\Gamma_n$-invariants of $X \in \{M^{\Psi=1}/\nabla,M_{\CC_p}^{\Psi=1}/\nabla\}$  can be expressed as the kernel of $$X \mapsto \prod_{i=1}^d X$$ where the map is given by $x \to (\gamma_i-1)x$ and $(\gamma_i)_i$ are topological generators of $\Gamma_n.$ Again, by the commutativity of $D$, this map is a $D$-linear map between coadmissible modules and hence strict. We can again apply \cite[Lemma 1.12]{steingart2024iwasawa}   and use finite dimensionality of $H^0(\Gamma_n,X)$ to obtain $$\CC_p \otimes_K H^0(\Gamma_n,M^{\Psi=1}/\nabla)\cong H^0(\Gamma_n,M_{\CC_p}^{\Psi=1}/\nabla).$$ The claim now follows from the Five Lemma. 

\end{proof}
\begin{lem} \label{lem:firstEPCformula}
Let $M$ be a $K$-analytic $(\varphi_K,\Gamma_K)$-module over $\cR_{\mathbb{C}_p}$  and let $\chi_{an}(\Gamma_n,M)$ be the Euler--Poincaré characteristic of $H^\bullet_{an}(\Gamma_n,M),$ then 
$$\operatorname{rank}_D(\mathcal{C}(M)) = -\chi_{an}(\Gamma_n,M).$$
\end{lem}
\begin{proof}
By \cite[Proposition 6.2]{steingart2024iwasawa}  the Euler--Characteristic is equal to the additive inverse of the fibre rank $\dim_{\CC_p}(\mathcal{C}(M)/Z\mathcal{C}(M)),$ where $Z \in D$ generates the augmentation ideal. By Lemma \ref{lem:finitepres} $\mathcal{C}$ is a finitely generated projective module over a domain and hence its fibre rank is constant and equal to its rank, which proves the claim. 
\end{proof}

\begin{thm} \label{thm:EPCformel}
Let $M$ be a $K$-analytic $(\varphi_K,\Gamma_K)$-module over $\cR_K$. Then 
$C_{\Psi}(M)$ is a perfect complex of $D^K(\Gamma_n,K)$-modules. Moreover
$M^{\Psi=1}$ is a finitely presented $D^K(\Gamma_n,K)$-module 
and \[
\operatorname{rank}_{D^K(\Gamma_n,K)} M^{\Psi=1} = -\chi_{an}(\Gamma_n,M) = [\Gamma_K:\Gamma_n]\operatorname{rank}(M).
\]
\end{thm}
\begin{proof}
$M^{\Psi=1}$ is finitely presented by Corollary \ref{cor:perfectpsi} and its rank is equal to the rank of $\mathcal{C}(M).$ Let $J = \Hom(K,\CC_p)\setminus \{\id\}.$ 
Applying \cite[Lemma 1.9]{Berger} to a surjection $D^r  \to M^{\Psi=1}$ shows that the kernel of the surjection is finite projective i.e. $M^{\Psi=1}$ is perfect. 
Using the same Lemma and finite dimensionality of $M/(\Psi-1)$ one obtains that $M/(\Psi-1)$ is perfect, and hence $C_{\Psi}(M)$ is perfect.
It remains to compute the rank of $M^{\Psi=1}$, for which we write \[
D_{\CC_p}:=D^K(\Gamma_n,\CC_p)
          =\CC_p\widehat{\otimes}_K D,
\qquad
M_{\CC_p}:=\CC_p\widehat{\otimes}_K M.
\]
There are canonical isomorphisms
\[
D_{\CC_p}\otimes_D M^{\Psi=1}
 \cong \CC_p\widehat{\otimes}_K M^{\Psi=1}
 \cong (M_{\CC_p})^{\Psi=1}.
\]
For the first isomorphism, take a finite presentation  \[
D^r \to D^s\to M^{\Psi=1}\to 0.
\] 
This sequence is strict exact for the coadmissible topologies.
Completed base change to $\CC_p$ is exact by \cite[Lemma 1.12]{steingart2024iwasawa} and yields 
a finite presentation of $\CC_p\widehat{\otimes}_K M^{\Psi=1}=(M_{\CC_p})^{\Psi=1}$ over $D_{\CC_p}$.
On the other hand, algebraic base change of the same presentation to $D_{\CC_p}$ gives 
a finite presentation of $D_{\CC_p}\otimes_D M^{\Psi=1}$, but both presentations are given by the same matrices over $D_{\CC_p}$, 
hence the two modules are canonically isomorphic. Since $D\hookrightarrow D_{\CC_p}$ is an injective map of domains, we see that 
the \[
\operatorname{rank}_D M^{\Psi=1} =\operatorname{rank}_{D_{\CC_p}}     (M_{\CC_p})^{\Psi=1}.
\]
Then, by \cite[Theorem 1]{steingart2025partial} the Euler--Poincaré Characteristic of $H^\bullet_{g,J}(K_n,W(M))$ is $-[K_n:K]\operatorname{rank}(M)$ 
and it is equal to the Euler--Poincaré Characteristic $\chi_{an}(\Gamma_n,M)$ by Theorem 1 (vii) \textit{(loc.cit.)}. 
Using Lemma \ref{lem:anbasechange} the result for $M_{\CC_p}$ now follows from Lemma \ref{lem:firstEPCformula}, which proves the claim.
\end{proof}

While the formula from Lemma \ref{lem:firstEPCformula} can be proven intrinsically 
(working only in the world of $(\varphi_K,\Gamma_K)$-modules) we require additional inputs to determine the precise value of 
$\chi_{an}(\Gamma_n,M),$ which is reminiscent of the Euler--Poincaré Characteristic formula proven in \cite{liu2007cohomology}, 
which relies on Tate's formula as a non-trivial input. While Herr in \cite{herr1998cohomologie} gives an independent proof, 
it relies on mod $p$ techniques which are not available over the Robba ring. Our input comes from the study of partial de Rham Selmer groups in 
\cite{steingart2025partial}, which use the Euler--Poincaré-Characteristic formula for Galois cohomology as an input. 
Because two representations with different Hodge--Tate weights can have the same reduction mod $p$ these do not seem to make sense in this setting. 
It remains unclear to us, whether there is an intrinsic proof of the Euler--Poincaré-Characteristic formula. 

\subsection{Comparison between $M^{\Psi=1}/\nabla$ and partial Selmer groups}
As before we fix $\Gamma_n \cong o_K$ and we denote by $\mathcal{D}_\infty$ the algebra of smooth distributions. Locally constant functions are automatically $\QQ_p$ and $K$-analytic, hence  $\mathcal{D}_\infty$ can be viewed as a quotient of both $\mathcal{D}^{\QQ_p}(\Gamma_n,K)$ and $\mathcal{D}^{K}(\Gamma_n,K).$ 
The kernel of $\mathcal{D}^{K}(\Gamma_n,K) \to \mathcal{D}_\infty$ is a principal ideal generated by a single element (namely a $K$-basis $\nabla$ of the Lie-algebra). 
\begin{lem} \label{lem:invariantscoinvariant}
Let $X$ be a finitely generated coadmissible $\mathcal{D}_\infty$-module and let $\Gamma' \subset \Gamma_n$ be an open subgroup.  There is an isomorphism
$$H^0(\Gamma',X) \cong X/\Gamma'.$$
\end{lem}
\begin{proof}
We can without loss of generality assume that $\Gamma'=\Gamma_n$ by viewing $X$ as a coadmissible $\mathcal{D}^{K}(\Gamma',K)$-module.
Since $X$ is coadmissible we can write $X= \varprojlim_m X_m$ with Banach spaces $X_m$ such that the image of $X$ in $X_m$ is dense. Since $\nabla$ acts trivially on $X$ the resulting action on $X_m$ is discrete. This can be seen by expanding $\gamma = \operatorname{exp(\log((\gamma-1)+1)}$ for $\gamma$ close to $1.$
Because we are in characteristic zero, the invariants and coinvariants of $X_m$ are isomorphic. 
We obtain 
\[
H^0(\Gamma_n,X) = \varprojlim_m H^0(\Gamma_n,X_m) \cong \varprojlim_m X_m/(\Gamma_n) \cong X/\Gamma_n,
\]
where in the last step we apply Theorem B to the submodule of $X$ generated by the augmentation ideal, which is also coadmissible because it is a finitely generated submodule of a coadmissible $\mathcal{D}^{K}(\Gamma_n,K)$-module. 
\end{proof}
\begin{prop} Let $W$ be a $K$-analytic $B$-pair and $M = M(W)$
The natural map 
$$M^{\Psi=1}/\nabla \to \varprojlim_{m}H^1_{an}(\Gamma_m,M) \cong \varprojlim_{m}H^1_{an}(K_m,W)$$ is injective and its co-kernel is finite $K$-dimensional. In particular both sides have $D_\infty$-rank equal to $[\Gamma_K\colon\Gamma_n]\operatorname{rank}(W).$
\end{prop}
\begin{proof}
As in the proof of Lemma \ref{lem:anbasechange}
we have a short exact sequence 

$$0 \to H^0(\Gamma_m,M^{\Psi=1}/\nabla) \to H^1_{an}(\Gamma_m,M) \to H^0(\Gamma_m,[M/(\Psi-1)M]^{\nabla=0}) \to 0.$$

Since by Lemma \ref{cor:perfectpsi} $M^{\Psi=1}$ is finitely presented over $D$ we have, by the same arguments as in the proof of Proposition \ref{prop:DmodNablaLimit},
$M^{\Psi=1}/\nabla \cong \varprojlim_m M^{\Psi=1}/(\Gamma_m).$ Because  $M^{\Psi=1}/\nabla$ is a $\mathcal{D}_\infty$-module Lemma \ref{lem:invariantscoinvariant} implies that coinvariants and invariants are isomorphic. Furthermore, the $H^0(\Gamma_m,M^{\Psi=1}/\nabla)$ are finite $K$-dimensional since they inject into a finite $K$-dimensional space. By Mittag-Leffler
passing to the limit with respect to $m$ yields an exact sequence 
$$0 \to M^{\Psi=1}/\nabla \to \varprojlim_m H^1_{an}(\Gamma_m,M) \to X \to 0.$$ 
where $X= \varprojlim_m H^0(\Gamma_m,[M/(\Psi-1)M]^{\nabla=0})$ is finite dimensional because the dimension of each term is bounded independently of $m.$ The claim now follows from the rank computation in Proposition \ref{prop:smoothrank}.
\end{proof}

\subsection{Dual Comparison map}

The term ``admissible representation'' is quite overloaded, for example it has different meanings in \cite{schneider2003algebras}, \cite{schneiderteitlbaumlocallyanalytic}.
Let $\Lambda:= o_K\llbracket \Gamma_K \rrbracket,$ $\Lambda_K:= \Lambda[1/p].$ Let $\mathcal{D}^{\QQ_p}(\Gamma_K,K)$ be the $\QQ_p$-analytic distribution algebra and $\mathcal{D}^K(\Gamma_K,K)$-the $K$-analytic distribution algebra (which is a quotient of $\mathcal{D}^{\QQ_p}(\Gamma_K,K)$ by the Ideal $J$ generated by the Cauchy-Riemann differential equations for $K/\QQ_p$ cf. \cite{schneider2001p}).
\begin{defn}
	We say a locally convex topological vector space $V$ over $K$ with a continuous action of $\Gamma_K$ is $\Lambda_K$-admissible, if $V'=\Hom_{cts}(V,K)$ is a finitely generated $\Lambda_K$-module. We say $V$ is $\mathcal{D}^{\QQ_p}$-admissible (resp. $D^K$-admissible), if $V$ is of compact type, and the strong dual $V'_b$ is a coadmissible $\mathcal{D}^{\QQ_p}(\Gamma_K,K)$-module with respect to the unique $\mathcal{D}^{\QQ_p}(\Gamma_K,K)$-module structure extending the natural $\Lambda_K$-module structure (resp. if $V$ is of compact type and $V'_b$ is a coadmissible $\mathcal{D}^K(\Gamma_K,K)$-module.)
\end{defn} 
\begin{ex}
	\label{ex:H1admissible}
	Let $V \in \operatorname{Rep}_K(G_K),$ then 
	$W:= H^1(H_K,V)$ is $\Lambda_K$-admissible, $W'$ is finitely presented over $\Lambda_K$ and 
	$(W^{la})'_b = \mathcal{D}^{\QQ_p}(\Gamma_K,K)\otimes_{\Lambda_K}W'$ is coadmissible, i.e., $W^{la}$ is $D^{\QQ_p}$-admissible. 
	Furthermore $(W^{K-la})'_b= \mathcal{D}^K(\Gamma_K,K)\otimes_{\Lambda_K}W',$ in particular, $W^{K-la}$ is $\mathcal{D}^K$-admissible. 
\end{ex}
\begin{proof}
	The first part follows (using Tate-duality) from $W'=H^1_{\text{Iw}}(K_\infty/K,V^*(1))$ being finitely presented by \cite{SV15}. The claim for $M:=(W^{la})'_b$ is \cite[Proposition 7.1]{schneider2003algebras}. For the last part note that $W^{K-la}$ is the subspace of $W^{la}$ annihilated by the ideal $J.$
	This ideal is finitely generated, hence $JM \subset M$ is a closed (hence coadmissible) sub module and dualising the strict exact sequence 
	$$0 \to JM \to M \to M/JM \to 0$$ allows us to identify $M/JM = \mathcal{D}^{\QQ_p}(\Gamma_K,K)/J \otimes_{\mathcal{D}^{\QQ_p}(\Gamma_K,K)}M $ with the dual of the closed subspace $$X = \{w \in W^{la} \mid \langle w,m \rangle= 0 \text{ for all } m \in JM \}.$$ We clearly have $W^{K-la} = W^{J=0} \subset X.$ Conversely, if $x \in X$ and $m \in M,$ then $\langle x, jm\rangle =\langle \iota(j)x,m \rangle = 0,$ where $\iota$ is the automorphism induced by inversion on $\Gamma_K.$ Hence $\iota(j)x=0$ as a functional on $M,$ but by reflexivity $\iota(j)x=0.$ Note that $J$ is stable under $\iota,$ from which we conclude $x \in W^{K-la}.$
	Combining the above 
	\begin{align*}(W^{K-la})'_b &= D/J \otimes_{\mathcal{D}^{\QQ_p}(\Gamma_K,K)} M \\
	&= D^K(\Gamma_K,K) \otimes_{\mathcal{D}^{\QQ_p}(\Gamma_K,K)} (\mathcal{D}^{\QQ_p}(\Gamma_K,K) \otimes_{\Lambda_K}W')\\
	&= \mathcal{D}^K(\Gamma_K,K)\otimes_{\Lambda_K}W'.\end{align*}
	
\end{proof}
Let $H_K \subseteq G_K$ be the kernel of the Lubin--Tate character. 
In \cite{SV23comparing} the relationship between different categories of $(\varphi_K,\Gamma_K)$-modules is explained in detail. 
Recall that to $V \in \operatorname{Rep}_K(G_K)$ we can assign 
$\tilde{D}^\dagger_{rig}(V):= (\widetilde{\cR}\otimes_KV)^{H_K}$ a $(\varphi_K,\Gamma_K)$-module over the perfect Robba ring $\widetilde{\cR}.$ We also have $\tilde{D}^\dagger(V)$ defined over the sub-ring of overconvergent elements. Furthermore, if $V$ is overconvergent, then $\tilde{D}^\dagger_{rig}$ (resp. $\tilde{D}^\dagger$) can be descended to $D^\dagger_{rig}(V)$ and resp. $D^\dagger(V)$ living over $\cR_K$ (resp. the overconvergent subring $\cR_K^\dagger.$) 
Let us recall the main results namely:
\begin{thm}
	Let $V \in \operatorname{Rep}_\dagger(G_K),$ then for $\tilde{D} = \tilde{D}^\dagger_\text{rig}(V)$ or $\tilde{D}^\dagger(V)$ respectively, $D =D^\dagger_\text{rig}(V)$ or $D^\dagger(V)$ we have a canonical $\Gamma_K$-equivariant commutative diagram with exact rows
$$	\begin{tikzcd}
		0 \arrow[r] & {H^0(H_K,V)} \arrow[r] & \tilde{D} \arrow[r, "\varphi-1"] & \tilde{D} \arrow[r] & {H^1(H_K,V)} \arrow[r] & 0 \\
		0 \arrow[r] & {H^0(H_K,V)} \arrow[r] \arrow[u, equal] & D \arrow[r, "\varphi-1"] \arrow[u, hook] & D \arrow[r] \arrow[u, hook] & {H^1_\dagger(H_K,V)} \arrow[u, hook] \arrow[r] & 0,
	\end{tikzcd}$$
	where $H^1(H_K,V)$ is the Galois cohomology of $H_K$ with coefficients in $V$ and $H^1_\dagger(H_K,V)$ is defined as the co-kernel of $\varphi-1$ on $D.$
\end{thm}
In \cite[Remark 5.9]{SV23comparing} it is further shown that for $K\neq \QQ_p$ (and $V \neq 0$) the inclusion $H^1_\dagger(H_K,V) \to H^1(H_K,V)$ is not an isomorphism. 
We will make this more precise, it is in fact never an isomorphism.
\begin{prop} \label{prop:Qpproper} Let $0 \neq V \in \operatorname{Rep}_\dagger(G_K).$
	The natural map $H^1_\dagger(H_K,V) \to H^1(H_K,V)$ is not surjective. The image $H^1_\dagger(H_K,V)$ is contained in the (proper) subspace $H^1(H_K,V)^{la}$ of $\QQ_p$-locally analytic vectors. 
\end{prop}
\begin{proof}
	It is an easy observation that the elements of $D = D^\dagger(V)$ or $D^\dagger_{rig}(V)$ have pro-analytic $\Gamma$-action. It follows that the action of $\gamma$ on their images in $H^1(H_K,V)$ is pro-analytic. Since $H^1(H_K,V)$ is a Banach space, this is the same as being analytic. The dual of $H^1(H_K,V)$ is $H^1_{\text{Iw}}(K_\infty/K,V^*(1)).$ In particular, finitely generated over the Iwasawa algebra $o_K\llbracket \Gamma_K \rrbracket[1/p].$ Consequently $H^1(H_K,V)$ is admissible, and thus, the locally analytic vectors are a dense subspace, which can be equipped with a LB-topology which makes the inclusion $H^1(H_K,V)^{la} \to H^1(H_K,V)$ continuous. If this map were an isomorphism, it would have to be a homeomorphism by the open mapping theorem. But in this case the linear duals would coincide, but since $W:=H^1(H_K,V)$ is admissible, we have: 
	$(W^{la})' = \mathcal{D}^{\QQ_p}(\Gamma) \otimes_{\Lambda} W'.$ By the assumption $V \neq 0$ we conclude that $W'$ has non-zero $\Lambda$-rank. It is easy to see that $W' \neq \mathcal{D}^{\QQ_p}(\Gamma) \otimes_{\Lambda} W'.$
\end{proof}
\begin{prop}
	Let $K=\QQ_p,$ then $H^1_\dagger(H_K,V) = H^1(H_K,V)^{la}.$
\end{prop}
\begin{proof}
We use again that $H^1(H_K,V)$ is an admissible representation, hence $(H^1(H_K,V)^{la})' = \mathcal{D}^{\QQ_p}(\Gamma)\otimes_{\Lambda} H^1_{\text{Iw}}(V^*(1)) = D^\dagger_{\text{rig}}(V^*(1))^{\Psi=1}$ by \cite[Proposition V.1.18.]{colmez2010representations}. 
Dualising and using that $\varphi-1$ is the transpose of $\psi-1$ yields the claim. 
\end{proof}
\begin{lem}
	Let $V \in \operatorname{Rep}^{K-an}(G_K).$ Then 
\begin{enumerate}
	\item 	$H^1_\dagger(H_K,V)$ is a subspace of $H^1(H_K,V)^{K-la}.$
	\item 	 $(H^1(H_K,V)^{K-la})' \cong D^{K-an}(\Gamma)\otimes_{\Lambda}H^1_{\text{Iw}}(V^*(1)).$
\end{enumerate}
\end{lem}
\begin{proof}
The first part is proven analogously as \ref{prop:Qpproper}. The second part is Example \ref{ex:H1admissible}.
\end{proof}

\begin{thm} Suppose $V^*(1)$ is $K$-analytic. Let $D:=D^\dagger_{\text{rig}}(V^*(1)),$ let $\chi_C$ be the character $x\abs{x}$ and $\check{D}:= \Hom(D,\cR_K(\chi_C)).$ Then
the transpose of $$D/(\varphi-1)D= H^1_\dagger(H_K,V^*(1)) \xrightarrow{\alpha} H^1(H_K,V^*(1))$$ defines a surjective map of coadmissible $D^K(\Gamma_K,K)$-modules 
$$\operatorname{comp} \colon \mathcal{D}^{K} \otimes_{\Lambda}H^1_{\text{Iw}}(K_\infty/K,V) \to \check{D}^{\Psi=1}.$$
\end{thm}
\begin{proof}
Recall from \cite[Lemma 4.5.1]{SchneiderVenjakobRegulatorBook} that we have a sesquilinear residue pairing
$D \times \check{D} \to K$ which allows us to identify $\check{D}^{\Psi=1}$ with the subspace of continuous functionals on $D$ vanishing on $(\varphi-1)D.$ 
By injectivity of $\alpha$, the image of $\alpha'$ is dense in the weak star topology (cf. \cite[Theorem 7.4.9]{PGS}); this is also the weak topology of $X:=\check{D}^{\Psi=1}.$ If the closure $Y$ of $\operatorname{Im}(\alpha')$ was not all of $X,$ then $X/Y$ would admit a non-zero functional, contradicting \cite[Theorem 5.2.5]{PGS}. We conclude that the image of $\alpha'$ is dense. Finally a map of coadmissible modules is automatically strict, and hence has closed image. 
\end{proof}
In \cite{steingart2024iwasawa}  it was shown that the surjectivity of the comparison map would imply the coadmissibility of the $\Psi=1$ eigenspace and it was believed that it could serve as an input to prove coadmissibility. In this present note we took a different approach and proved coadmissibility directly and it turns out that this coadmissibility is in fact equivalent to the surjectivity of the comparison map.

\let\stdthebibliography\thebibliography
\let\stdendthebibliography\endthebibliography
\renewenvironment*{thebibliography}[1]{%
	\stdthebibliography{BHH+22}}
{\stdendthebibliography}

\bibliographystyle{amsalpha}

\bibliography{Literatur}
\end{document}